\documentclass[11pt]{article}
\usepackage{amsmath, amsthm, amssymb, amsfonts}
\usepackage{graphicx}
\usepackage{hyperref}
\usepackage{bm}
\usepackage{float}
\usepackage{bbm}
\usepackage{xcolor}

\usepackage[backend=bibtex,maxbibnames=50,style=alphabetic,sorting=nyt,giveninits=true]{biblatex}
\newtheorem{theorem}{Theorem}
\newtheorem{proposition}[theorem]{Proposition}
\newtheorem{lemma}[theorem]{Lemma}
\newtheorem{corollary}[theorem]{Corollary}

\theoremstyle{definition}
\newtheorem{definition}{Definition}
\newtheorem{remark}{Remark}

\newcommand\ddfrac[2]{\frac{\displaystyle #1}{\displaystyle #2}}
\newcommand{\st}{\mathcal{S}}
\newcommand{\A}{\mathcal{A}}
\newcommand{\opn}{\mathcal{O}}
\def\R{{\mathbb R}}
\def\P{{\mathbb P}}

\begin{document}

\title{
Metastability and phase transition in a social network model with multiple opinions}
\author{Felipe Penafiel\thanks{E-mail: felipe.lev.pena@gmail.com}  \\\and Kádmo Laxa\thanks{Faculdade de Filosofia, Ciências e Letras de Ribeirão Preto, Universidade de São Paulo, Av. Bandeirantes, 3900,
Ribeirão Preto-SP, 14040-901, Brazil.
 E-mail: kadmolaxa@gmail.com}}
\date{\today}
\maketitle

\begin{abstract} 
We consider a stochastic opinion dynamics model on a fully connected social network with $N$ actors interacting by expressing opinions from a set of $M$ opinions. At any time $t\geq 0$, each actor is associated to an $M$-tuple representing the social pressure exerted on this actor for each opinion. The evolution of the matrix containing the social pressure of all actors for all opinions is a Markov jump process. Each actor tends to express opinions according to their social pressure vector and this tendency is modulated by a polarization coefficient. When an actor expresses an opinion $o$, its social pressure for all opinions is reset to zero, while for other actors the social pressure for $o$ increases by 1 and the social pressure for other opinions decreases by $1/(M-1)$. 
In this setting, we prove fast consensus formation, existence of a unique invariant measure and metastability in a highly polarized network. Moreover, by considering a communication bias parameter, the system exhibits a phase transition described as follows. With a negative communication bias parameter, all actors except one stop expressing in a finite time almost surely. Otherwise, no actor stops expressing opinions.
 \vspace{0.2cm}
 
 
 \vspace{0.2cm}
 
\end{abstract}

\section{Introduction}

The present paper studies a probabilistic social dynamics model on a fully connected social network with $N\geq 3$ actors interacting by expressing opinions from a set of $M\geq 2$ possible opinions. We describe it informally as follows.

At any time $t\geq 0$, each actor is associated with a $M$-tuple representing the social pressure exerted on this actor for each opinion. The evolution of the matrix containing the social pressure of all actors for all opinions is a Markov jump process. 
The rate at which an actor expresses a certain opinion grows exponentially with its social pressure on the opinion. The speed at which actors express opinions and the tendency of each actor to express opinions in which the actor has the greatest social pressure increase with a non-negative parameter named \textit{polarization coefficient}.

When an actor expresses an opinion, its social pressure for all opinions is reset to zero and the social pressure of all the other actors is modified in the following way. The social pressures of the other actors for the expressed opinion increase by $1$ and the social pressures of the other actors for the opinions different from the expressed one decrease by $1/(M-1)$.

Let us now informally present our results. Consider the stochastic process describing the time evolution of the matrix containing the social pressure of each actor for each opinion in a network with $N$ actors, $M$ opinions and a fixed polarization coefficient. The existence of the process and the uniqueness of its invariant probability measure are the content of Theorem \ref{Theorem 1}. 

A consensus matrix of social pressures is a non-null matrix in which the social pressures of all actors for a given opinion are non-negative and the social pressures of all actors for the other opinions are non-positive. The invariant probability measure concentrates on the set of consensus matrices, and the hitting time to this set vanishes almost surely as the polarization coefficient diverges (to be more precise, the concentration of the invariant measure and the fast hitting time occur to a subset of the consensus matrices named ladder sets, which will be defined in Section \ref{sec:modelandresults}). This is the content of Theorem \ref{Theorem 2}.

As the polarization coefficient diverges, the system takes a long and unpredictable random time to go from a consensus matrix in the direction of a certain opinion to a consensus matrix in the direction of any other opinion. This means that the change from consensus to consensus in the network has metastable behavior. This is the content of Theorem \ref{Theorem 3}.

In our next result, we consider a new parameter $\alpha\in\mathbb{R}$ in our model representing the \textbf{communication bias}. We modify the model presented above in the following way: when an actor expresses an opinion, its social pressure for all opinions is reset to zero and the social pressure of all the other actors is modified in the following way. 
The social pressures of the other actors on the expressed opinion increase by $1$ and the social pressures of the other actors on the other opinions increase by $-1/(M-1)+\alpha$. Note that if $\alpha> \frac{1}{M-1}$, any opinion expressed in the network yields positive social pressure for all actors, preventing the formation of consensus. Therefore, our next result deals with the cases $\alpha < \frac{1}{M-1}$.

We have a phase transition: when $0\leq \alpha< \frac{1}{M-1}$, adapted versions of the results described above for the model without communication bias hold. In particular, no actor stops expressing opinions. However, if $\alpha< 0$, eventually all the actors, except one, stop expressing opinions in the system. This is the content of Theorem \ref{Theorem 4}.

The model studied in the present article is a generalization of the one introduced in \cite{galves2024fastconsensusmetastabilityhighly}. In fact, the model of \cite{galves2024fastconsensusmetastabilityhighly} is equivalent to the model introduced here in the case $M=2$. The results of the present article dealing with the properties of the invariant measure, the fast formation of consensus, and metastability are generalizations of the results presented in \cite{galves2024fastconsensusmetastabilityhighly}. However, the communication bias parameter is a novel feature of the present article, which allows the study of phase transition. The model  introduced by \cite{galves2024fastconsensusmetastabilityhighly} is also considered in \cite{evalaxa}, which studies mean field limits as the number of actors on the network increases and considers more general functions describing the jump rates of the system.

Since each actor resets its social pressures when expressing an opinion, the rate at which an actor expresses each opinion depends on the activity of the network since the last time the actor expressed an opinion. Therefore, this model is a system of interacting point processes with memory of variable length. More precisely,
our model belongs to the  class of systems of interacting point processes with memory of variable length that was introduced in discrete time by \cite{glmodel} and in continuous time by \cite{gl4} to model systems of spiking neurons. We refer to \cite{gb} for the presentation of this class of models and a full list of references.  

The notion of metastability considered here is inspired by the \textit{pathwise approach to metastability}
introduced by \cite{cassandro}. 
For more references and an introduction to the topic, we refer the reader to \cite{metabook}.
In particular, metastability by the pathwise approach for systems of spiking neurons belonging to the class introduced by \cite{glmodel} and \cite{gl4} was studied by \cite{mandre1,mandre2,taille,evamonm,mandre3,laxa}. For other works in this class of models studying phase transition, we refer to  \cite{brochiniphasetransition,gl2, amarcos}. We refer to \cite{galves2024fastconsensusmetastabilityhighly}
for the motivation to consider a social network model with the features of this class of models.

Opinion dynamics models describe how individual states evolve through social interaction. The voter model \cite{Liggett1985} is the closest antecedent to our framework: agents on a network update their discrete opinion by imitating others, and the competition between imitation and noise drives the system between consensus and coexistence. Recent multi-opinion extensions \cite{HerrerasAzcue2019} have shown that this competition gives rise to noise-driven transitions between unimodal and multimodal stationary distributions, closely related to our metastability results. A parallel tradition models opinions as continuous values updated by averaging: the DeGroot model \cite{DeGroot1974} and its bounded confidence variants \cite{Deffuant2000, HegselmannKrause2002} focus on how the topology of influence determines whether consensus or opinion clustering emerges. Our model differs from both traditions in a key respect: rather than updating opinions directly, actors accumulate social pressure over time and express opinions at rates driven by this pressure, making the model a system of interacting point processes with variable-length memory. We refer to \cite{Peralta2022} for a broader survey of the field.

The paper is organized as follows: In Section~\ref{sec:modelandresults} we define the model and state the main theorems for the model without considering the communication bias parameter.  In Section~\ref{sec:communicationbias} we define the model with communication bias and state the phase transition result. In Section \ref{sec:visualintuition} we present some simulations to illustrate the behavior of the model. In Section~\ref{sec:proofs} we prove the main results. Finally, in Section~\ref{sec:discussion} we discuss extensions and perspectives.

\section{Model and main results} \label{sec:modelandresults}

For $N\geq 3$, let $\mathcal{A}=\{1,...,N\}$ be a set representing \textbf{actors} and $\mathcal{O}=\{1,...,M\}$ be a set representing \textbf{opinions}. The model consists of an $N\times M$ matrix evolving over time in order to describe the \textbf{social pressure} of each actor for each opinion. 

In the following, any matrix of social pressure will be denoted as $u=(u(a,o))_{a\in\mathcal{A},\,o\in\mathcal{O}}$. The entry in row $a$ and column
$o$ is denoted $u(a,o)$. We use $u(a,\cdot)$ to denote the entire row
vector corresponding to actor $a$ (social pressures of this actor toward all
opinions), and $u(\cdot,o)$ to denote the column vector corresponding
to opinion $o$ (social pressures of all actors toward this opinion). 

When an actor $a\in\mathcal{A}$ expresses opinion $o\in\mathcal{O}$, we apply the following \textbf{operator} on the matrix $u$, describing the social pressure values:
\begin{equation}\label{modelmap}
    [\pi^{a,o}(u)](b,p)=
    \begin{cases}
        0 & \text{if $b = a$}, \\
        u(b,p) + 1 & \text{if $b\neq a$ and $p=o$}, \\
        u(b,p) - \frac{1}{M-1} & \text{if $b\neq a$ and $p\neq o$}.
    \end{cases}
\end{equation}
Thus, the process will be restricted to the following set
\begin{multline}
    \label{set_S}
    \mathcal{S} = \{u = (u(a,o))_{a\in\mathcal{A}, o\in\mathcal{O}}\in\mathcal{M}_{N,M}(\mathbb{R}): \\ \forall a\in\mathcal{A},o\in\mathcal{O}, u(a,o)\in\mathbb{Z}+(1/(M-1))\mathbb{Z};\\\min_{a\in\mathcal{A}}\|u(a,\cdot)\|_{\infty}=0;\forall a\in\mathcal{A},\sum_{o\in\mathcal{O}}u(a,o)=0 \}.
\end{multline} 
Here, $\mathcal{M}_{N,M}(\mathbb{R})$ denotes the set of real matrices with $N$ rows and $M$ columns and $(1/(M-1))\mathbb{Z}=\{(1/(M-1))z:z\in \mathbb{Z}\}$ . Indeed, from any starting position in $\mathcal{M}_{N,M}(\mathbb{R})$, after all actors express an opinion at least once, the process will remain in $\st$ forever.

The rates at which the actors express opinions on the network depend on a parameter $\beta\geq0$, we shall call it the \textbf{polarization coefficient}. We use the notation $(U_t^{\beta,u})_{t\geq0}$ in order to define the time evolution of the social pressure matrix with polarization coefficient $\beta\geq0$ and initial matrix $U_0^{\beta,u}=u$. As a \textbf{Markov Jump Process}, its evolution over time is described by the \textbf{generator}
\begin{equation} \label{generator:withoutbias}
\mathcal{G}f(u)=\sum_{o\in\mathcal{O}}\sum_{a\in\mathcal{A}}\exp(\beta u(a,o))[f(\pi^{a,o}(u))-f(u)],    
\end{equation}
where $f:\st \to \R$ is a bounded function. Moreover, we set $(T_n,A_n,O_n)_n$ to respectively represent the expression times, actors and opinions driving the evolution of the random process.

We now present the main theorems. Proofs are deferred to Section~\ref{sec:proofs}.

The first theorem states that the random process is defined for any positive time $t$, and that it has a unique invariant probability measure.

\begin{theorem}[Existence]\label{Theorem 1}
    For any $\beta\geq 0$ and any starting matrix $u\in\mathcal{S}$:
    \begin{enumerate}
        \item The sequence $(T_m:m\geq 1)$ of jumping times of the process $(U^{\beta, u}_t)_t$ satisfies
    \[\mathbb{P}(\sup\{T_m:m\geq 1\} = \infty)=1. \]
        \item The process $(U^{\beta,u}_t)_t$ has a unique invariant probability measure $\mu^\beta$.
    \end{enumerate}
\end{theorem}

For the second theorem, we need to define ladder sets. Those sets have a sort of stability property in the following sense: whenever the most probable choice of actor and opinion occurs, the process remains on the ladder set. This favored opinion concentrates social pressure in a specific way, while other opinions all have negative social pressure.

\begin{definition}\label{newL} 
    For any given opinion $o\in\mathcal{O}$, we define the ladder sets supporting opinion $o$ as follows:
\begin{multline*}
\mathcal{L}^o=\{u\in\st: \{u(1,o),...,u(N,o) \} = \{0,...,N-1 \}, \\ u(\cdot,p)=\frac{-u(\cdot,o)}{M-1}, \forall p \in\mathcal{O}\setminus \{o\}\}.
\end{multline*}
    Also, the ladder set is the union of the ladder sets on all of the $M$ opinions:
    \[\mathcal{L}=\bigcup_{o\in\mathcal{O}}\mathcal{L}^o.\]
\end{definition}

The second theorem states that the invariant probability measure concentrates in ladder sets, and that the ladder set hitting time vanishes as $\beta\rightarrow+\infty$. We define the hitting time for any $\beta\geq0, u\in\mathcal{S}, \theta \subset\mathcal{S}$:

$$
    R^{\beta,u}(\theta):=\inf\{t\geq0:U^{\beta,u}_t\in\theta\}.
$$

\begin{theorem}[Concentration of the invariant measure]\label{Theorem 2} The following statements about ladder sets are true:
        \begin{enumerate}
            \item There exists a constant $C>0$, such that for any $\beta\geq 0$ the invariant probability measure $\mu^{\beta}$ satisfies
            \[\mu^{\beta}(\mathcal{L})\geq 1-Ce^{-\frac{\beta}{M-1}}. \]
            \item For any fixed $\delta>0$
            \[\sup_{u\in\mathcal{S}\setminus\{0\}}\mathbb{P}\left(R^{\beta, u}(\mathcal{L})>e^{-\frac{\beta}{M-1}(1-\delta)} \right)\rightarrow 0\text{, as }\beta\rightarrow +\infty. \]
    \end{enumerate}
\end{theorem}

For the third theorem, we introduce consensus sets. This is a larger set of states than ladder sets, since we only impose the favored opinion to have non-negative social pressure on all actors, while the social pressure for other opinions is non-positive.

\begin{definition}[Consensus sets]\label{newCminus}

For any opinion $o\in\mathcal{O}$, a consensus set for $o$ is defined in the following way
\begin{equation}\label{C}
    \mathcal{C}^o = \{u \in \mathcal{S}:u\neq 0, u(a,o)\geq 0 ,u(a,p)\leq 0, \forall a\in\mathcal{A}, \forall p\in\mathcal{O}\setminus\{o\}\}.
\end{equation}
Note that for any $u\in \mathcal{C}^o$, there exists at least one actor $a\in \A$ such that $u(a,o)>0$.
    A consensus set for another opinion than $o\in\mathcal{O}$ is defined as follows:
    \[
    \mathcal{C}^{-o}=\bigcup_{p\in\mathcal{O}\setminus\{o\}}\mathcal{C}^p .
    \]
\end{definition}

Finally, the third theorem establishes the metastable behavior when we look at the time it takes to observe a consensus transition when $\beta\rightarrow+\infty$.

\begin{theorem}[Metastability]\label{Theorem 3}
There exists $\beta_0, C_1>0$ and  $C_2 \in (0,1/2)$, depending only on $M$ and $N$, such that for any $\beta\geq \beta_0$, for any $o\in\mathcal{O}$ and for any  $u\in \mathcal{C}^o$,
$$
\sup_{t\geq 0}\left|\P\left(\frac{R^{\beta,u}(\mathcal{C}^{-o})}{\mathbb{E}(R^{\beta,u}(\mathcal{C}^{-o}))}>t\right)-e^{-t}\right| \leq C_1\beta^3e^{-C_2 \beta}.
$$
Moreover, for any $u,v \in \mathcal{C}^o$,
$$
\left|\frac{\mathbb{E}(R^{\beta,u}(\mathcal{C}^{-o}))}{\mathbb{E}(R^{\beta,v}(\mathcal{C}^{-o}))}-1\right| \leq C_1\beta^3e^{-C_2 \beta}.
$$
\end{theorem}

Note that the theorem above implies that for any $o\in\mathcal{O}$ and for any $u\in \mathcal{C}^o$,
\[\frac{R^{\beta, u}(\mathcal{C}^{-o})}{\mathbb{E}\left[R^{\beta, u}(\mathcal{C}^{-o}) \right]}\rightarrow Exp(1)\text{, in distribution, as }\beta\rightarrow + \infty, \]
where $Exp(1)$ denotes a random variable following an exponential probability law of parameter 1.

\section{Model with communication bias and phase transition result} \label{sec:communicationbias}
Let us now consider a new model.
We introduce a new parameter $\alpha\in\mathbb{R}$. It represents a \textbf{communication bias} and defines a new family of operators:
\begin{equation}\label{biasmap}
    [\pi_\alpha^{a,o}(u)](b,p)=
    \begin{cases}
        0 & \text{if $b = a$,} \\
        u(b,p) + 1 & \text{if $b\neq a$ and $p=o$,} \\
        u(b,p) - \frac{1}{M-1}+\alpha & \text{if $b\neq a$ and $p\neq o$.}
    \end{cases}
\end{equation}
For this new model, when an actor $a\in\mathcal{A}$ expresses opinion $o\in\mathcal{O}$, we apply the operator $\pi_\alpha^{a,o}$.

Let us write $\gamma:=\frac{1}{M-1}-\alpha$. In this model, expressing opinion $o$ adds $1$ to the social pressure of every other actor for $o$ and adds $-\gamma$ to its social pressure for every other opinion. As a consequence, the row of an actor $a$ that has heard $n_a$ expressions since the last time it expressed an opinion, $c_p$ of which concerned opinion $p$, equals $c_p(1+\gamma)-\gamma n_a$ in column $p$. Since the actors cannot express opinions simultaneously, note that $\sum_{a\in \A}n_a \geq 0+1+\ldots +(N-1)=N(N-1)/2$. The process is therefore restricted to the following set.
\begin{multline} \label{biased_set_S}
\st^{\alpha}= \Bigl\{u= (u(a,o))_{a\in\mathcal{A}, o\in\mathcal{O}} \in \mathcal{M}_{N,M}(\mathbb{R}):\min_{a\in\mathcal{A}}\|u(a,\cdot)\|_{\infty}=0; \\ 
\forall a\in\mathcal{A}, \exists\, n_a\in\mathbb{N}, \sum_{a\in \A}n_a \geq N(N-1)/2, \ \exists\,(c_p)_{p\in\opn}\in\mathbb{N}^M \\
\text{ with } \sum_{p\in\opn}c_p=n_a\text{ and } u(a,p)=c_p(1+\gamma)-\gamma n_a \text{ for all } p \in \opn \Bigr\}.
\end{multline}
 
\begin{remark}\label{rem:Salphaproperties}
The set $\st^{\alpha}$ is stable under every operator $\pi^{a,o}_{\alpha}$: hearing an expression of $p$ increases $n_a$ and $c_p$ by one, and expressing resets the row of the expressing actor to $(n_a,c)=(0,0)$. Moreover, it implies adapted versions of the constraints of definition (\ref{set_S}): $\sum_{p\in\opn}u(a,p)=(1+\gamma)n_a-M\gamma n_a=n_a(M-1)\alpha$ for every $a\in\A$, and $u(a,p)=c_p+\gamma(c_p-n_a)\in\mathbb{Z}+\gamma\mathbb{Z}$. Finally, a row is null if and only if $n_a=0$, since a balanced row with $n_a\geq 1$ has all its entries equal to $\frac{n_a(M-1)\alpha}{M}\neq 0$ when $\alpha\neq 0$. Finally, note that $0\not \in \st^{\alpha}$.
\end{remark}

We use the notation $(U_t^{\alpha,\beta,u})_{t\geq0}$ in order to define the time evolution of the social pressure matrix with polarization coefficient $\beta>0$, communication bias $\alpha \in \R$ and initial matrix $U_0^{\alpha, \beta,u}=u$. As a \textbf{Markov Jump Process}, its evolution over time is described by the \textbf{generator}
\begin{equation}\label{biasGenerator}
    \tilde{\mathcal{G}}f(u)=\sum_{o\in\mathcal{O}}\sum_{a\in\mathcal{A}}\exp(\beta u(a,o))[f(\pi_\alpha^{a,o}(u))-f(u)],
\end{equation}
where $f:\st^{\alpha} \to \R$ is a bounded function.

Moreover, we set $(T_n^{\alpha,\beta,u},A_n^{\alpha,\beta,u},O_n^{\alpha,\beta,u})_n$ to respectively represent the expression times, actors and opinions driving the evolution of the random process.

For the case $0<\alpha< \frac{1}{M-1}$, so $\gamma>0$, we can adapt the ladder and consensus sets in the following way.
\begin{equation}\label{def:consensussetalpha}
\mathcal{C}_{\alpha}^{o} = \{u \in \mathcal{S}^{\alpha}:u\neq 0, u(a,o)\geq 0 ,u(a,p)\leq 0, \forall a\in\mathcal{A}, \forall p\in\mathcal{O}\setminus\{o\}\},
\end{equation}
\begin{multline}
\label{def:laddersetalpha}
\mathcal{L}_{\alpha}^{o}=\Big\{u\in\st^{\alpha}: \{u(1,o),...,u(N,o) \} = \{0,...,N-1 \}, \\ u(\cdot ,p)=-\gamma u(\cdot ,o), \forall p \in\mathcal{O}\setminus \{o\}\Big\}.
\end{multline}
We analogously define 
 \[
    \mathcal{C}^{-o}_{\alpha}=\bigcup_{p\in\mathcal{O}\setminus\{o\}}\mathcal{C}^p_\alpha
    \]
and
 \[\mathcal{L}_{\alpha}=\bigcup_{o\in\mathcal{O}}\mathcal{L}^o_{\alpha}.\]

\begin{theorem}[Phase transition]\label{Theorem 4} The distinct cases of positive and negative bias yield the following results.
    \begin{enumerate}
        \item For any $\beta>0$ and for any $\alpha<0$, we have that for any $u\in\mathcal{S}^{\alpha}$ 
        $$
            \mathbb{P}\left(\bigcup_{n\geq1}\bigcap_{m\geq n} \{A^{\alpha,\beta,u}_n=A^{\alpha,\beta,u}_m\} \right)=1.
        $$
        \item For any $0\leq \alpha <  \frac{1}{M-1}$, adapted versions of Theorems \ref{Theorem 1}, \ref{Theorem 2} and \ref{Theorem 3} that consider the model \eqref{biasGenerator} and the definitions of ladder and consensus sets given by \eqref{def:consensussetalpha} and \eqref{def:laddersetalpha} hold. 
    \end{enumerate}
\end{theorem}

The adapted versions of Theorems \ref{Theorem 1}, \ref{Theorem 2} and \ref{Theorem 3} are Theorems \ref{Theorem 1_2}, \ref{Theorem 2_2} and \ref{Theorem 3_2}, presented in Appendix \ref{ap:part2theorem4}.

\begin{remark}  In the case $\alpha>\frac{1}{M-1}$ we have that any expression yields positive social pressure towards all opinions. With this, the notion of consensus introduced in the case $\alpha<\frac{1}{M-1}$ cannot be considered. Moreover, when $\alpha\geq1+\frac{1}{M-1}$, we have a degenerate situation in which any expression yields more social pressure towards opinions other than the one supported by the expression. See Section \ref{sec:visualintuition} for a discussion concerning these cases.
\end{remark}

\section{Some visual intuition} \label{sec:visualintuition}

Beyond the analogy to social dynamics, intuition on this model can also be built thanks to visualizations generated by means of numerical simulation. We will thus indulge in some exploration of this sort, before heading into the proofs of our theorems in the next section. 

Figure \ref{fig:fig1} gives insight into Theorem \ref{Theorem 2} by giving a clear picture of what a ladder set looks like. In the plot, actors are ordered so that their pressure values increase step by step --- from $0$ up to $ N - 1 $ --- producing the staircase profile predicted by the theorem. The key point is that actors’ pressures for a given opinion are aligned on this shape, while pressure for other opinions remains non-positive. As the polarization parameter $\beta$ grows, the invariant probability measure is locked into such sets. Indeed, note that Figure \ref{fig:fig1} illustrates the following stability property: An expression by the actor with the biggest social pressure (high probability when $\beta$ is big) moves the process to another ladder supporting the same opinion.

\begin{figure}[H]
    \centering
    \includegraphics[width=0.49\textwidth]{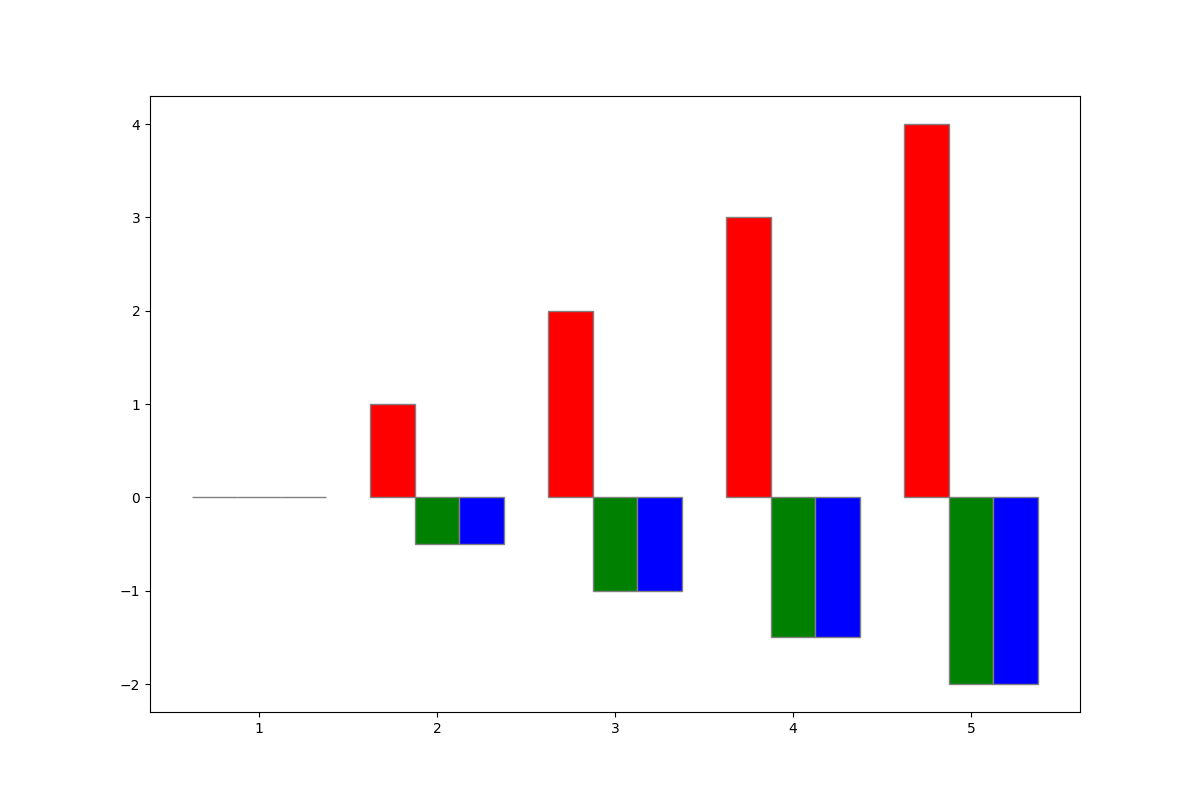}
    \includegraphics[width=0.49\textwidth]{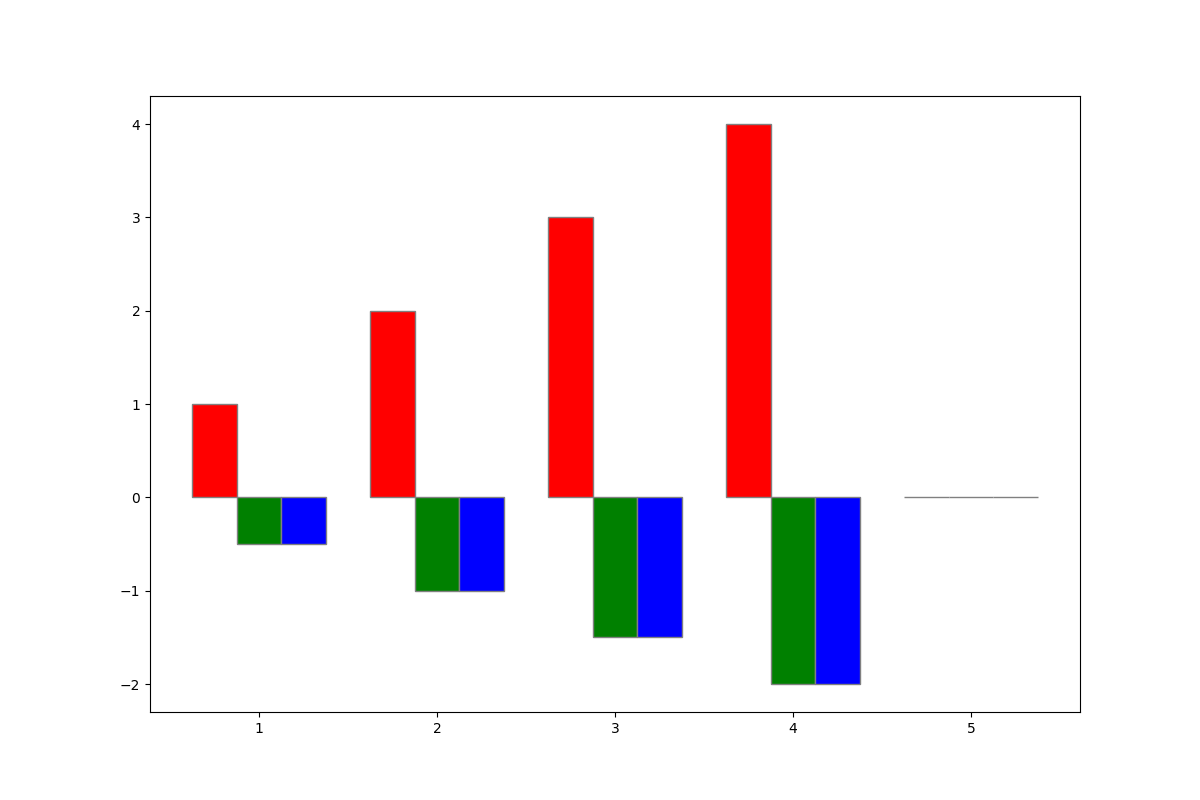}
    \caption{The representation of social pressures of two different elements of the ladder sets for $N=5$ and $M=3$. The red, green and blue bars represent the social pressures of each actor for the opinions $1,2$ and $3$, respectively. In the left figure, actor $1$ has social pressure $0$ for all opinions, actor 2 has social pressure $1$ for opinion $1$ and $-1/2$ for opinions $2$ and $3$, actor 3 has social pressure $2$ for opinion $1$ and $-1$ for opinions $2$ and $3$, and so on.
    Note that the matrix of social pressures on the right figure is obtained from the matrix of social pressures on the left figure when actor 5 expresses opinion $1$, which is the most probable event.}
    \label{fig:fig1}
\end{figure}

To illustrate Theorems \ref{Theorem 3} and \ref{Theorem 4}, it is useful to go beyond the microscopic description in terms of social pressures $u(a,o)$. We thus introduce aggregate macroscopic observables of social dynamics.

For each opinion $o \in \mathcal{O}$ and configuration
$u \in \mathcal{S}$, the \textbf{public opinion} toward $o$ is defined as
the total pressure exerted by the population:
\[
    P_o(u) \;=\; \sum_{a \in \mathcal{A}} u(a,o).
\]

Figure \ref{metastableConsensusM=4} provides an illustration of Theorem \ref{Theorem 3}. The plot shows the evolution of the public opinion vector in the case $M=4$. One observes long time intervals during which a single opinion dominates and the system remains close to a consensus set. These plateaus correspond to the predicted metastable regimes. Eventually, abrupt transitions occur, and a different opinion becomes dominant, leading the system to a new metastable consensus. The alternation between stability over long periods and sudden switches exemplifies the behavior formalized by the metastability theorem.

\begin{figure}[H]
    \centering
    \includegraphics[width=\textwidth]{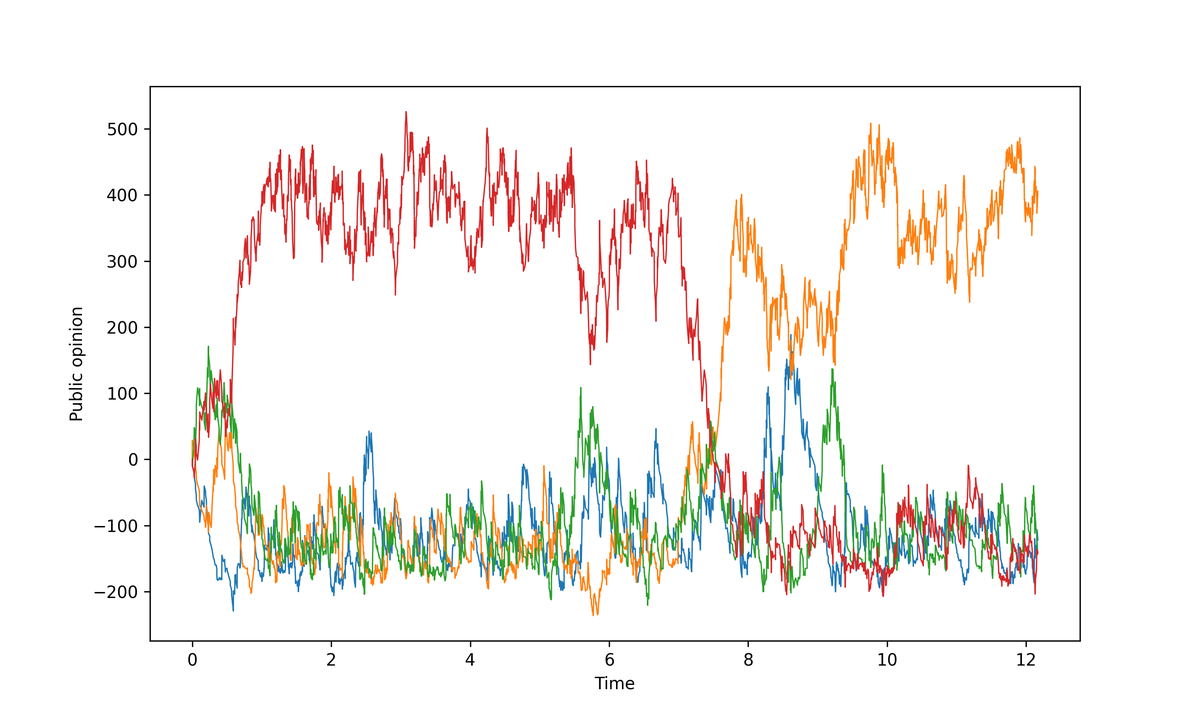}
    \caption{Time evolution of the public opinion vector in the case $M=4$. Each line represents the public opinion towards each of the $4$ opinions. In this simulation, in a first moment the opinion associated with the red color become dominant, describing a moment in which the social pressure of the actors are in general positive for this opinion and the majority of the expressed opinions in the network follow this direction. Suddenly, this situation is replaced by the dominance of the opinion associated with the orange color, illustrating the metastable behavior of the system.}
    \label{metastableConsensusM=4}
\end{figure}

Now, we turn to Theorem \ref{Theorem 4}. Actually, it was mainly through numerical simulations with different values of $\alpha$ that we first caught a glimpse of a phase transition. And we wish to share some of the logical progression of our exploration. First, let us quickly describe the case $\alpha\geq 0$.  

For values of $0\leq\alpha<1+\frac{1}{M-1}$, actors are still able to express some preference for the opinion they support, so a dominant opinion may still emerge. Note that the notion of consensus needs to be adapted to the state space $\st^{\alpha}$, but the main features are the same. When $\alpha\geq1+\frac{1}{M-1}$, we have a degenerate situation in which any expression yields more social pressure towards opinions other than the one supported by the expression. This prevents the formation of any lasting consensus. 

We now move on to the more interesting case of negative communication bias. In Figure \ref{fig:lossofconsensus}, we still observe transitions of consensus as Theorem \ref{Theorem 3} would predict in the case $\alpha=0$, but at some point the dynamics change. 
\begin{figure}[H]
    \centering
    \includegraphics[width=\textwidth]{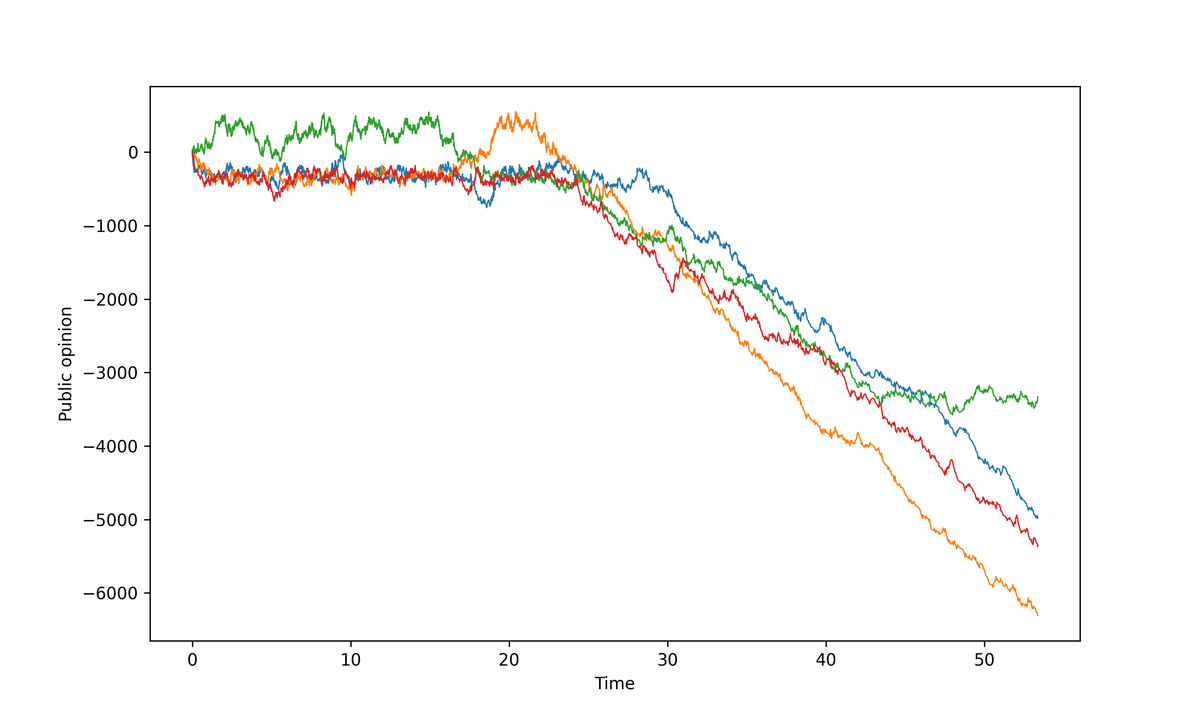}
    \caption{Time evolution of the public opinion vector in the case $M=4$ and $\alpha=-\frac{1}{M-1}$. Each line represents the public opinion towards each of the $4$ opinions. In this simulation, in a first moment the system exhibits a metastable behavior similar to the one displayed in Figure \ref{metastableConsensusM=4}. Suddenly, this situation is replaced by the decay of the public opinion for all opinions.}
    \label{fig:lossofconsensus}
\end{figure}

We may observe two distinct behaviors in this figure. In the first half, we see the formation of consensus, although the transition to a new consensus is not as sharp as in Figure \ref{metastableConsensusM=4}. However, in the second half of the graph, something changes. There is a new period without consensus, except here no opinion manages to recover before the end of the simulation. Instead, opinions suddenly start a negative cycle in which their public opinion decreases over time. At this stage, it is not yet clear what is causing this phenomenon. 

To get a better grasp of the situation, we shall define a new macroscopic observable. For each actor $a \in \mathcal{A}$ and configuration $u \in \mathcal{S}$,
the \textbf{trust} of actor $a$ is defined as the total pressure allocated
by this actor to all possible opinions:
\[
    T_a(u) \;=\; \sum_{o \in \mathcal{O}} u(a,o).
\]

\begin{remark}
Before the addition of communication bias, we have $T_a(u)=0$ for all $a \in \mathcal{A}$, so that trust vanishes identically. This reflects the normalization originally built into the model, ensuring that each actor’s support for some opinions is exactly offset by reduced support for the others. 
\end{remark}

The new dynamics at play is one where the actors enter an irreversible spiral of negativity. Their constantly decreasing social pressure for all opinions can make them stop expressing opinions permanently. 
The time evolution of the actors' trust is displayed in Figure \ref{fig:lossoftrust}. Theorem \ref{Theorem 4} provides a
complete understanding of this phenomenon, as it states that the observed behavior is indeed irreversible.

\begin{figure}[H]
    \centering
    \includegraphics[width=\textwidth]{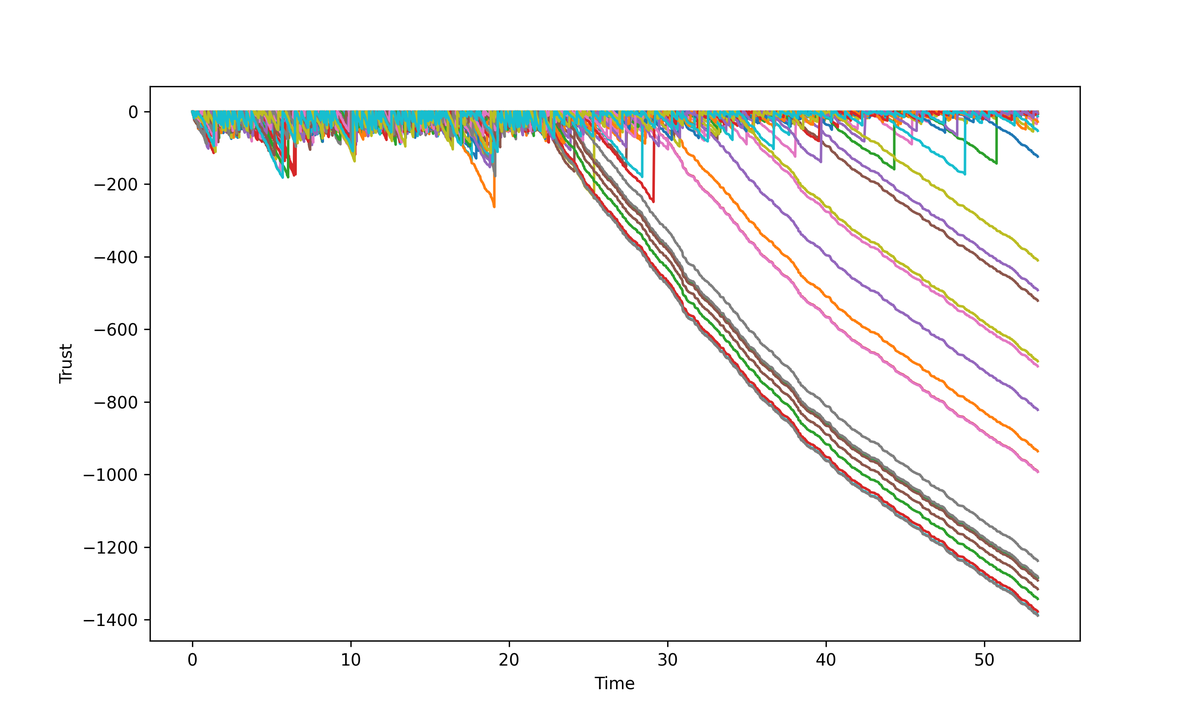}
    \caption{Time evolution of the trust in the case $M=4$ and $\alpha=-\frac{1}{M-1}$. Each line represents the trust of each one of the $N=30$ actors. The simulation is the same as the one illustrated in Figure \ref{fig:lossofconsensus}, revealing that the decay of public opinion is caused by the successive
    loss of trust of the actors.}
    \label{fig:lossoftrust}
\end{figure}

\section{Proofs} \label{sec:proofs}

Let us begin this section with some initial results that describe the basic behavior of the system. The following proposition is the basis of the proof of Part 1 of Theorem \ref{Theorem 1}. We show that among each $N$ consecutive opinions expressed in the system, at least one comes from an actor with low social pressure on the expressed opinion.

\begin{proposition}\label{prop1}
    For any starting matrix $u\in\st$, we have:
    $$
        \inf\{n\geq 1: \|U^{\beta,u}_{T_{n-1}}(A_n,\cdot)\|_\infty<N \} \leq N.
    $$
\end{proposition}

\begin{proof}
    Suppose by contradiction that $\|U^{\beta,u}_{T_{n-1}}(A_n,\cdot)\|_\infty\geq N$ for all $n=1,\ldots, N$. The initial matrix of social pressures $u$ belongs to $\st$. Therefore, there exists $a_0\in\mathcal{A}$ such that $u(a_0,\cdot)=0$. Thus, by definition, for any $m\geq 1$:
    $$
        \|U^{\beta,u}_{T_m}(a_0,\cdot) \|_\infty \leq m.
    $$
    This implies that $A_n\neq a_0$, for all $n=1,\ldots,N$. By definition $U^{\beta,u}_{T_1}(A_1,\cdot) =0$. For the same reasons described above, we conclude that $A_n\neq A_1$, for all $n=2,\ldots,N$. Iterating this step, we are led to conclude that the set $\{a_0,A_1,\ldots, A_N\}$ is made of $N+1$ distinct elements which is impossible, since $|\A|=N$. With this we conclude the proof.
\end{proof}

Let us consider the following definition. For a fixed $u\in\st$ and $n\geq 1$, we define:
    $$
        \xi_n^u:=\{(A_n,O_n)\in\arg\max_{(a,o)\in\mathcal{A}\times\mathcal{O}}U^{\beta,u}_{T_{n-1}}(a,o)\}. 
    $$
In event $\xi_n^u$, the system with initial social pressure $u$ chooses the $n$-th pair of actor/opinion in the most likely way. This means that $(A_n,O_n)$ is one of the pairs that maximize the social pressure at time $T_{n-1}$. The next two propositions describe the consequences on the system of the successive occurrence of these events.

\begin{proposition} \label{prop2}
For any $u\in \st$, the following holds
$$
 \bigcap_{j=1}^N\xi_j^u\subset\{-MN<U^{\beta,u}_{T_N}<N\}.
$$
\end{proposition}

Here and in what follows, inequalities between a matrix and a real number are understood entrywise.

\begin{proof}
First note that if $|\{A_1, \ldots, A_N\}|=N$, 
by using the same logic as in Proposition \ref{prop1}, we have that $\|U^{\beta,u}_{T_N}\|_\infty<N$. Moreover, if $|\{A_1, \ldots, A_N\}|< N$ and $\bigcap_{j=1}^N\xi_j^u$ occurs, then there exists a time $m \in \{1, \ldots, N\}$ such that an actor repeats itself, so
$$
U_{T_{m-1}}^{\beta,u}(A_m,o) < m-1
$$
and actually
$$
U_{T_{m-1}}^{\beta,u}(a,o) < m-1
$$
for all $a \in \A$ and $o\in \opn$. This implies that
$$
\max\{U^{\beta,u}_{T_N}(a,o):a\in \A,o\in \opn\} < N, 
$$
and since $\sum_{o\in \opn}U^{\beta,u}_{T_N}(a,o)=0$, for any $a\in \A$, we conclude the proof.
\end{proof}

\begin{proposition}\label{prop3}

For any initial matrix $u\in\st$, 
$$
  \left\{\bigcap_{j=1}^{(M+1)N}\xi_j^u\right\} \subset \left\{U^{\beta,u}_{T_{(M+1)N}}\in\mathcal{L} \right\}.
$$
\end{proposition}

The proof follows the main ideas of Part 2 of Proposition 5 of \cite{galves2024fastconsensusmetastabilityhighly}. By considering this more general model, we have more technical difficulties. We present this proof in Appendix \ref{appendix1}.

The next proposition gives a lower bound for the probability of the successive occurrence of the events $\xi_n^u, n\geq 1$. The proof is similar to the proof of Proposition 5 of \cite{galves2024fastconsensusmetastabilityhighly}.

\begin{proposition}\label{prop4}
    For all $ m\geq 1,$
    $$ \mathbb{P}\left(\bigcap^m_{j=1}\xi_j^u\right)\geq (\zeta_{\beta})^m,
    $$
    where
    $$
        \zeta_{\beta}=\frac{e^{\frac{\beta}{M-1}}}{e^{\frac{\beta}{M-1}}+MN}\rightarrow1\text{, as }\beta\rightarrow+\infty.
    $$
\end{proposition}

\begin{proof}
    We first remark that
    \begin{equation}\label{eq1prop11}
\mathbb{P}\left(\bigcap^m_{j=1}\xi_j^u\right)=\sum_{v\in\st}\mathbb{P}\left(\bigcap^{m-1}_{j=1}\xi_j^u,U^{\beta,u}_{T_{m-1}}=v\right)\mathbb{P}\left(\xi_m^u|\bigcap_{j=1}^{m-1}\xi_j^u, U^{\beta,u}_{T_{m-1}}=v\right).
\end{equation}

In the following, let us obtain a lower bound for $\mathbb{P}\left(\xi_m^u|U^{\beta,u}_{T_{m-1}}=v\right)=\P(\xi_1^v)$, for any $v\in \st$. Note that in the case $v=0$, the probability $\P(\xi_1^v)$ is maximized.
For any $v \in \st \setminus\{0\}$,
$$
\P(\xi_1^v)=\frac{|Y(v)|e^{\beta y(v)}}{|Y(v)|e^{\beta y(v)}+ \displaystyle\sum_{b\in \A}\ \sum_{\quad o\in \opn:\ v(b,o)< y(v)}e^{\beta v(b,o)}},
$$
where
$
y(v)=\max\{v(a,o): a\in \A, o\in \opn\}
$
and
$
Y(v)=\{(a,o) \in \A\times \opn: v(a,o)=y(v)\}.
$ 
From this point on, note that $|Y(u)|\geq 1$ and that, by \eqref{set_S}, every entry of $v$ lies in $\mathbb{Z}+\frac{1}{M-1}\mathbb{Z}=\frac{1}{M-1}\mathbb{Z}$. Hence, for any $b \in \A$ and $o\in\opn$ such that $v(b,o)< y(v)$, the difference $v(b,o)-y(v)$ is a strictly negative element of $\frac{1}{M-1}\mathbb{Z}$, so that $v(b,o)-y(v)\leq -\frac{1}{M-1}$. We finish the proof by rearranging the terms.
\end{proof}

\begin{remark} \label{remark:bernoulli}
By putting Proposition \ref{prop4}  together with Bernoulli's inequality, we have that for any $m\geq 1$ and for any $\beta\geq 0$
$$
(\zeta_{\beta})^m \geq 1-mMNe^{\frac{-\beta}{M-1}}.
$$
\end{remark}

\subsection{Proof of Theorem \ref{Theorem 1}} \label{subsec:T1}

To prove part 1 of this theorem, we will use the following strategy. We shall divide opinion expressions into two kinds. Those that come from actors with high social pressure, and the others coming from actors with low social pressure. Proposition \ref{prop1} tells us that opinion expressions of the second type must happen at least once every $N$ opinion expressions.

\begin{proof}
The jump times $\{T_n\}_n$ of process $(U^{\beta,u}_t)_t$ are the superposition of the jump times $\{T^<_n\}_n$, in which the actor expressing an opinion has social pressure smaller than $N$, and $\{T^>_n\}_n$, in which the actor expressing an opinion has social pressure greater than or equal to $N$. The jump times $\{T^<_n\}_n$ can be constructed in such a way that 
\begin{equation} \label{eq:jumpbounds}
\{T^{(M)}_n\}_n \subset\{T^<_n\}_n \subset \{T^{(\lambda)}_n\}_n,   
\end{equation}
where $\{T^{(M)}_n\}_n$ and $\{T^{(\lambda)}_n\}_n$ are the jump times of homogeneous Poisson point processes with rate $M$ and $\lambda:=NMe^{\beta N}$, respectively. Since 
$\mathbb{P}(\sum_{n\geq 1} \mathbbm{1}_{\{T_n^{(\lambda)} <t\}} <\infty)=1$ for any $t>0$, we have that $\mathbb{P}(\sum_{n\geq 1} \mathbbm{1}_{\{T_n^{<} <t\}} <\infty)=1$.
By Proposition \ref{prop1}, this implies that for any $t>0$, 
$\mathbb{P}(\sum_{n\geq 1} \mathbbm{1}_{\{T_n^> <t\}} <\infty)=1$.
Now, $\{T_n\}_n$ being the superposition of $\{T^<_n\}_n$ and $\{T^>_n\}_n$, we conclude that $\mathbb{P}(\sum_{n\geq 1} \mathbbm{1}_{\{T_n <t\}} <\infty)=1$.
Finally, considering the lower bound in \eqref{eq:jumpbounds}, we have that $\sup \{T^<_n\}_n=+\infty$. This allows us to conclude the proof.
\end{proof}

The construction described above is a simple modification of the one presented in the proof of part $1$ of Theorem 1 of \cite{galves2024fastconsensusmetastabilityhighly}.

We now prove part 2 of Theorem \ref{Theorem 1}.

\begin{proof}
Let $o\in\mathcal{O}, l_o\in\mathcal{L}^o$ where,
$$
\forall a\in\mathcal{A},p\in\mathcal{O}\setminus\{o\},
$$
$$
    l_o(a,o)=a-1,l_o(a,p)=-\frac{a-1}{M-1}.
$$
We remark that 
$$
    \forall u \in \st, l_o=\pi^{1,o}\circ \pi^{2,o}\circ\cdots\circ\pi^{N,o}(u).
$$
This implies that
\begin{equation} \label{escadalo}
    \mathbb{P}(U_{T_N}^{\beta,u}=l_o)\geq \P\left(\bigcap_{j=1}^N\{A_j=N-j+1,O_j=o\}\right)>0.
\end{equation}

We further argue that for any $u'\in\st$,
$$
    \mathbb{P}(U^{\beta,u}_{T_{n+2N}}=l_o|U^{\beta,u}_{T_n}=u')\geq
$$
$$
    \mathbb{P}(\|U^{\beta,u}_{T_{n+N}}\|_\infty<MN|U^{\beta,u}_{T_n}=u')\times\mathbb{P}(U^{\beta,u}_{T_{n+2N}}=l_o|U^{\beta,u}_{T_n}=u',\|U^{\beta,u}_{T_{n+N}}\|_\infty<MN).
    $$

By putting together Propositions \ref{prop2} and \ref{prop4}, we have that $\forall u\in\st$
$$
\mathbb{P}(\|U^{\beta,u}_{T_N}\|_\infty<MN)\geq\mathbb{P}(\bigcap_{j=1}^N\xi_j^u)\geq \zeta^N_\beta.
 $$
Denote
$$
\varepsilon^*:=\min\{\mathbb{P}(U^{\beta,v}_{T_{N}}=l_o):v\in\st,\|v\|_\infty<MN\}.
$$
Since the social pressures in the set $\{v\in\st:\|v\|_\infty<MN\}$ are bounded and by \eqref{escadalo}, the visit to the state $l_o$ is guaranteed by the successive occurrence of actors $1,2 \ldots, N$ expressing opinion $o$, we have that $\varepsilon^*>0$. Therefore,
$$
        \mathbb{P}(U^{\beta,u}_{T_{n+2N}}=l_o|U^{\beta,u}_{T_n}=u')\geq\zeta^N_\beta\varepsilon^*.
$$
This bound, uniform in $u'\in\st$, is a Doeblin condition. The skeleton chain $(U^{\beta,u}_{T_n})_n$ is therefore uniformly ergodic and
admits a unique invariant probability measure $\tilde{\mu}^\beta$.

For a non-explosive process, a probability measure is invariant for
$(U^{\beta,u}_t)_t$ if and only if its product with the jump rate is invariant
for the skeleton chain. Non-explosivity holds by Part~1, and since every
$u\in\st$ carries an actor with null social pressure, the jump rate obeys
$q_\beta(u)\geq M$ for all $u\in\st$. Hence
$$
    \sum_{u\in\st}\frac{\tilde{\mu}^\beta(u)}{q_\beta(u)}
    \;\leq\;\frac{1}{M}\sum_{u\in\st}\tilde{\mu}^\beta(u)
    \;=\;\frac{1}{M}\;<\;\infty,
$$
so that
\begin{equation}
    \mu^\beta(u)\;:=\;\frac{\tilde{\mu}^\beta(u)/q_\beta(u)}
    {\sum_{v\in\st}\tilde{\mu}^\beta(v)/q_\beta(v)},
    \qquad u\in\st, \label{eq:invariant_measure_def}
\end{equation}
is a well-defined probability measure, invariant for $(U^{\beta,u}_t)_t$. The
correspondence being a bijection between the stationary laws of the two
processes, $\mu^\beta$ is the unique invariant probability measure of
$(U^{\beta,u}_t)_t$, which proves Part~2 of Theorem~\ref{Theorem 1}.
    
\end{proof}

\subsection{Proof of Theorem \ref{Theorem 2}} \label{subsec:T2}

In order to prove Theorem \ref{Theorem 2}, we shall first prove a variant of it for the invariant measure of the skeleton process, which we define as follows.

\begin{definition}
For any $u\in \st$ and $\beta \geq 0$, the process $(\tilde{U}_n^{\beta,u})_{n=0,1,\ldots}$ given by $\tilde{U}_n^{\beta,u}:= U_{T_n}^{\beta,u}$ for any $n=0,1,\ldots,$ is the skeleton process associated to $(U_t^{\beta, u})_{t\geq 0}$. The unique invariant measure of this skeleton process is denoted $\tilde{\mu}^\beta$. Finally, we consider for any $\theta\subset \st$,
$$
\tilde{R}^{\beta,u}(\theta):=\inf\{n\geq 1: \tilde{U}_n^{\beta,u} \in \theta\}.
$$
\end{definition}

Proposition \ref{theo2prop1} proves that the invariant measure of the skeleton process gets concentrated in an extended set of steep ladders as $\beta \to \infty$. This set is defined as follows.

\begin{definition}\label{def:hatladder}
For any $o \in \mathcal{O}$, we define
\begin{multline*}
    \hat{\mathcal{L}}^{o} = \{ u \in \mathcal{S} : u(a,o) \in \mathbb{Z},\ \forall a \in \mathcal{A}, \\
    0 = u(a_1,o) < u(a_2,o) < \ldots < u(a_N,o),\ \{a_1,\ldots,a_N\} = \mathcal{A}, \\
    \text{and } \forall p \neq o,\ u(a,p) = -\tfrac{1}{M-1}\, u(a,o) \},
\end{multline*}
and we set $\hat{\mathcal{L}} = \bigcup_{o \in \mathcal{O}} \hat{\mathcal{L}}^{o}$.
\end{definition}

\begin{remark} \label{rem:hatladder}
The set $\hat{\mathcal{L}}$ is the set of states that can be reached by starting in $\mathcal{L}^{o}$, for $o\in \opn$, and sequentially expressing opinions $o$ in the network. With this, the social pressures in the direction of the opinion $o$ remain all different and are all integer numbers.
Note that $\mathcal{L}\subset \hat{\mathcal{L}}$ and for any $\hat{l} \in \hat{\mathcal{L}}$, 
$$\{U^{\beta,\hat{l}}_0(A_1,O_1)>0\} \subset \{ U^{\beta,\hat{l}}_{T_1} \in \hat{\mathcal{L}}\}.
$$
Moreover, for any $\hat{l}\in\hat{\mathcal{L}}$ and for a fixed $l\in \mathcal{L}$,
$$
\mathbb{P}(U^{\beta,\hat{l}}_0(A_1,O_1)>0)\geq \mathbb{P}(U^{\beta,l}_0(A_1,O_1)>0)\geq \eta,
$$
where,
    $$
        \eta = \ddfrac{\sum^{N-1}_{j=1}e^{\beta j}}{\sum^{N-1}_{j=0}e^{\beta j}+MN}.
    $$

\end{remark}

\begin{proposition}\label{theo2prop1}
     For any $\beta>0, u\notin\hat{\mathcal{L}}$, the following bound holds
    $$
        \Tilde{\mu}^\beta(u)\leq C'e^{-\beta(N-1)},
    $$
    where $C'=(NM)^{(M+1)N+1}>0$.
\end{proposition}

The proof of Proposition \ref{theo2prop1} is analogous to the proof of Proposition 11 of \cite{galves2024fastconsensusmetastabilityhighly}, so we provide only a sketch of the proof. 

\begin{proof}
Consider an initial matrix  $u\notin\hat{\mathcal{L}}$. By the classical Kac's Lemma, we first remark that
$$
\frac{1}{\Tilde{\mu}^\beta(u)}=\mathbb{E}[\Tilde{R}^{\beta,u}(u)]\geq \sum_{m\geq(M+1)N}\mathbb{P}(\Tilde{R}^{\beta,u}(u)\geq m).
$$
Therefore, all we have to do is to bound $\mathbb{P}(\Tilde{R}^{\beta,u}(u)\geq m)$ from below. To do this, we first consider Proposition \ref{prop3} to show that a sequence of events $\xi^u_{j}, j=1,\ldots, (M+1)N$ leads the process to $\mathcal{L}$, without visiting $u$, with a lower bounded probability. Moreover, by Remark \ref{rem:hatladder} we show that starting from $l\in \mathcal{L}$, a sequence of events $\{U^{\beta,l}_{T_{j-1}}(A_j,O_j)>0\}_{j\geq 1}$ keeps the process in $\hat{\mathcal{L}}$, thus without visiting $u \not\in \hat{\mathcal{L}}$. By considering the lower bounds given by Proposition \ref{prop3} and Remark \ref{rem:hatladder} we finish the proof. 
\end{proof}

The following corollary follows from Proposition \ref{theo2prop1} in an analogous way as Corollary 12 follows from Proposition 11 of \cite{galves2024fastconsensusmetastabilityhighly} by noting that $0\in \st$ can only be visited from a state in which only one actor has null social pressure for all opinions while all the other actors have social pressure $-1$ for one of the opinions and $1/(M-1)$ for all the other opinions.

\begin{corollary}\label{theo2coro1}
    $\Tilde{\mu}^\beta(0)\leq C''e^{-\beta(N-1+\frac{1}{M-1})}$ for all $\beta>0$, where $C''=(NM)C'>0.$
\end{corollary}

The proof of Theorem \ref{Theorem 2} is very similar to the proof of Theorem $2$ of \cite{galves2024fastconsensusmetastabilityhighly}. However, we write this proof fully for the sake of completeness.

\begin{proof}
    By recalling \eqref{eq:invariant_measure_def}, we start by pointing out that for any $u\in\st$, the invariant probability measure $\mu^\beta$ satisfies
    $$
        \mu^\beta(u)=\frac{\Tilde{\mu}^\beta(u)}{q_\beta(u)}\left( \sum_{u'\in\st}\frac{\Tilde{\mu}^\beta(u')}{q_\beta(u')} \right)^{-1},
    $$
    where $q_\beta$ is the jump rate of a given state of the process. For any $v\in\st$, we have
    $$
        q_\beta(v)=\sum_{(a,o)\in\mathcal{A}\times\mathcal{O}}e^{\beta v(a,o)}.
    $$
    By noting that $q_\beta(l)\leq MNe^{\beta(N-1)}$, for any $l\in \mathcal{L}$, we have that
    $$
        \mu^\beta(\mathcal{L})=\sum_{u\in\mathcal{L}}\frac{\Tilde{\mu}^\beta(u)}{q_\beta(u)}\left( \sum_{u'\in\st}\frac{\Tilde{\mu}^\beta(u')}{q_\beta(u')} \right)^{-1}\geq
        \ddfrac{1}{1+\frac{MNe^{\beta(N-1)}}{\Tilde{\mu}^\beta(\mathcal{L})}\sum_{u'\notin\mathcal{L}}\frac{\Tilde{\mu}^\beta(u')}{q_\beta(u')}}.
    $$
    We shall now obtain sufficient bounds for each term in this expression.
    By using Propositions \ref{prop3} and \ref{prop4} we obtain the bound
    $$
        \Tilde{\mu}^\beta(\mathcal{L})\geq\zeta^{(M+1)N}_\beta=\left(\frac{e^{\frac{\beta}{M-1}}}{e^{\frac{\beta}{M-1}}+MN} \right)^{(M+1)N}\geq (MN)^{-(M+1)N}.
    $$
    Observe that for any $u\neq 0, q_\beta(u)\geq e^{\frac{\beta}{M-1}}$.  For any $u\notin\mathcal{L}$ and $\max_{(a,o)\in \A\times \opn} u(a,o)< N$, we have that $u\notin\hat{\mathcal{L}}$, and therefore, by using Proposition \ref{theo2prop1} and Corollary \ref{theo2coro1},
    $$
    \frac{\Tilde{\mu}^\beta(u)}{q_\beta(u)}\leq C''e^{-\beta(N-1+\frac{1}{M-1})}. 
    $$
    Moreover, for any $u\notin\mathcal{L},$ with $\max_{(a,o)\in \A\times \opn} u(a,o)\geq N$, we have $q_\beta(u)\geq e^{\beta N}$, so
    $$
        \frac{\Tilde{\mu}^\beta(u)}{q_\beta(u)}\leq \Tilde{\mu}^\beta(u)e^{-\beta N}
    $$
    Since $|\{u\notin\mathcal{L},\max_{(a,o)\in \A\times \opn} u(a,o)< N\}|:=K(N,M)\in\mathbb{N}$ we combine all these bounds in order to conclude that
    $$
        \mu^\beta(\mathcal{L})\geq \frac{1}{1+Ce^{\frac{-\beta}{M-1}}}\geq 1-Ce^{\frac{-\beta}{M-1}},
    $$
    where $C=(NM)^{(M+1)N+1}[1+C''K(N,M)]$.
    
    This concludes the proof of part 1 of the theorem, we move on to part 2 now. We first note that for any $u\in\st\setminus\{0\}$, we have
    \begin{equation} \label{eq:part2theorem2}
        \mathbb{P}(R^{\beta,u}(\mathcal{L})>t)\leq \mathbb{P}\left(R^{\beta,u}(\mathcal{L})>t, \bigcap_{j=1}^{(M+1)N}\xi_j \right) + \mathbb{P}\left(\bigcup_{j=1}^{(M+1)N}(\xi_j)^c \right).
    \end{equation}
    Proposition \ref{prop4} gives us a bound for one term,
    $$
        \mathbb{P}\left(\bigcup_{j=1}^{(M+1)N}(\xi_j)^c \right) \leq 1-\zeta_\beta^{(M+1)N}.
    $$
    And Proposition \ref{prop3} gives us a bound for the other term,
    $$
        \mathbb{P}\left(R^{\beta,u}(\mathcal{L})>t, \bigcap_{j=1}^{(M+1)N}\xi_j \right) \leq \mathbb{P}\left(T_{(M+1)N}>t, \bigcap_{j=1}^{(M+1)N}\xi_j \right).
    $$
    Considering that the jump rate in any state other than a zero matrix is greater than to $e^{\frac{\beta}{M-1}}$, and that under $\cap_j\xi_j$ the process never visits $0$, we obtain
    $$
        \mathbb{P}\left(T_{(M+1)N}>t, \bigcap_{j=1}^{(M+1)N}\xi_j \right) \leq \mathbb{P}\left(\sum_{n=1}^{(M+1)N}E_n>t \right),
    $$
    where $(E_n)_n$ is an i.i.d. sequence of random variables following an exponential law with mean $\ddfrac{1}{e^{\frac{\beta}{M-1}}}$, therefore
    $$
        \mathbb{P}\left(\sum_{n=1}^{(M+1)N}E_n>t \right) \leq \mathbb{P}\left(\bigcup_{n=1}^{(M+1)N}\{E_n>\frac{t}{(M+1)N}\} \right)  \leq
    $$    
    $$   
        (M+1)N\exp\left(\frac{-e^{\frac{\beta}{M-1}}t}{(M+1)N}\right).
    $$
    We conclude the theorem by taking $t=e^{-\frac{\beta}{M-1}(1-\delta)}$ and noting that the previous bounds do not depend on $u$.
\end{proof}

\begin{corollary}\label{theo2coro2}
    For any fixed $\delta>0$, 
    $$
        \mathbb{P}(R^{\beta,0}(\mathcal{L})>\tau+e^{-\frac{\beta}{M-1}(1-\delta)})\rightarrow0\text{ as }\beta\rightarrow+\infty,
    $$
    where $\tau$ is an exponentially distributed random variable with mean $\frac{1}{MN}$ independent from $(U^{\beta,u}_t)_t$.
\end{corollary}

\begin{remark}\label{theo2rem1}
Note that the first term on right-hand side of Equation \eqref{eq:part2theorem2} could be replaced by
$$
\mathbb{P}(R^{\beta,u}(\mathcal{L})>t,R^{\beta, u}(\mathcal{L}^o)<R^{\beta,u}(\mathcal{C}^{-o})).
$$
Therefore, by following the same steps of the proof of part 2 of Theorem \ref{Theorem 2}, it follows that for any fixed $o\in\mathcal{O},\delta>0$
$$
\inf_{u\in\mathcal{C}^o}\mathbb{P}\left(R^{\beta, u}(\mathcal{L}^o)<\min\{e^{-\frac{\beta}{M-1}(1-\delta)},R^{\beta,u}(\mathcal{C}^{-o})\} \right)\rightarrow 1\text{, as }\beta\rightarrow +\infty.
$$
\end{remark}

\subsection{Proof of Theorem \ref{Theorem 3}} \label{subsec:T3}

Our metastability result uses Theorem 5.3 in  \cite{evamonm} which shows that, under certain conditions, the probability distribution of a rescaled exit time in a time-homogeneous strong Markov process is approximated by an exponential distribution. The following proposition states a consequence of this result for our model. To state the next proposition, consider the following notation.

For any $o\in\mathcal{O}, l\in\mathcal{L}^{o}$, 
let $c_{\beta,l}$ be the positive real number such that

$$
    \mathbb{P}(R^{\beta,l}(\mathcal{C}^{-o})>c_{\beta,l})=e^{-1}.
$$
Due to the symmetric properties of the process, it is clear that $c_{\beta,l}=c_{\beta,l'}$ for any pair of matrices $l$ and $l'$ belonging to $\mathcal{L}$. Therefore, in what follows we will omit to indicate $l$ in the notation of $c_\beta$. 

\begin{proposition} \label{prop:evamonm}
Consider a fixed $o\in \opn$.
For $\varepsilon_1, \varepsilon_2 >0$ and $s_1,s_2 >0$, with $\varepsilon_1 +\varepsilon_2 \leq 1/2$, consider the following assumptions. 
\begin{equation} \label{eq:propevamonm_assump1}
\varepsilon_1 \geq  \sup_{l\in \mathcal{L}^o}\P(R^{\beta,l}(\mathcal{C}^{-o}) \leq s_1),
\end{equation}
\begin{equation} \label{eq:propevamonm_assump2}
\varepsilon_2 \geq \sup_{u\in \mathcal{S}} \P(R^{\beta,u}(\mathcal{L}^o\cup\mathcal{C}^{-o}) > s_2).
\end{equation}

Consider constants $C,\delta>0$ such that
\begin{equation} \label{eq:propevamonm_assump3}
(s_2/c_{\beta}) \vee\varepsilon_2 \leq Ce^{-\delta \beta}.
\end{equation}
Finally, let $K,\theta>0$ satisfy
\begin{equation} \label{eq:propevamonm_assump4}
\sup_{u\in \mathcal{C}^o}\P(R^{\beta,u}(\mathcal{C}^{-o}) \leq s_2) \leq Ke^{-\theta \beta}.
\end{equation}
Then, there exists $\beta_0=\beta_0(C,\delta)>0$ such that if assumptions \eqref{eq:propevamonm_assump1}, \eqref{eq:propevamonm_assump2}, \eqref{eq:propevamonm_assump3} and \eqref{eq:propevamonm_assump4} hold for $\beta\geq \beta_0$, then there exists
$K'=K'(\delta,C,K,\theta)$ such that for any $u\in \mathcal{C}^o$ 
$$
\sup_{t\geq 0}\left|\P\left(\frac{R^{\beta,u}(\mathcal{C}^{-o})}{\mathbb{E}(R^{\beta,u}(\mathcal{C}^{-o}))}>t\right)-e^{-t}\right| \leq K'\beta^3e^{-\min(\delta/3,1/2,\theta) \beta}.
$$
Moreover, for any $u,v \in \mathcal{C}^o$,
$$
\left|\frac{\mathbb{E}(R^{\beta,u}(\mathcal{C}^{-o}))}{\mathbb{E}(R^{\beta,v}(\mathcal{C}^{-o}))}-1\right| \leq K'\beta^3e^{-\min(\delta/3,1/2,\theta) \beta}.
$$
\end{proposition}

The following lemma shows that condition \eqref{eq:propevamonm_assump2} holds with $s_2=2\beta$ and $\varepsilon_2= (M+1)^2N^2e^{-\beta /(M+1)N}$.

\begin{lemma} \label{theo3lemm1} For any $u \in \st$,
    $$
\mathbb{P}(R^{\beta,u}(\mathcal{L})>2\beta)
\leq  (M+1)^2N^2e^{-\beta /(M+1)N}.
$$
\end{lemma}
\begin{proof} 
By Corollary \ref{theo2coro2},
\begin{equation} \label{eq:lemma13}
\mathbb{P}(R^{\beta,0}(\mathcal{L})>2\beta)\leq \P(\tau>\beta)+\sup_{u\neq 0}\mathbb{P}(R^{\beta,u}(\mathcal{L})>\beta),
\end{equation}
where $\tau$ is an exponentially distributed random variable with mean $\frac{1}{MN}$.
Since for any $u\neq 0$,
$$
\mathbb{P}(R^{\beta,u}(\mathcal{L})>2\beta)\leq \sup_{u\neq 0}\mathbb{P}(R^{\beta,u}(\mathcal{L})>\beta),
$$
we just have to bound the terms in the right-hand side of \eqref{eq:lemma13}.
We have that  $\P(\tau>\beta)=e^{\frac{-\beta}{MN}}$ and by the proof of part 2 of Theorem \ref{Theorem 2}, we have that for any $u\neq 0$,
$$
\mathbb{P}(R^{\beta,u}(\mathcal{L})>\beta) \leq 1-\zeta_{\beta}^{(M+1)N}+(M+1)N\exp\left(\frac{-e^{\frac{\beta}{M-1}}\beta}{(M+1)N}\right).
$$
By recalling Remark \ref{remark:bernoulli} to bound $1-\zeta_{\beta}^{(M+1)N}$ and putting the inequalities above together, we conclude the proof.
\end{proof}

The next lemma and corollary give a lower bound to $c_{\beta}$ and show that the condition \eqref{eq:propevamonm_assump1} is satisfied.

\begin{lemma} \label{theo3lemm2}For any $t>0$, for any $o \in \opn$, the following holds.
\begin{enumerate}
    \item For any $l\in \mathcal{L}^o$,
    $$
    \P(R^{\beta,l}(\mathcal{C}^{-o})>t) \geq e^{-2tN^3(M+1)^3e^{\frac{-\beta}{M-1}}}.
    $$
    \item For any $u\in \mathcal{C}^{o}$,
    $$
\P(R^{\beta,u}(\mathcal{C}^{-o}) \leq t) \leq (N^2M+2tN^3(M+1)^3)e^{\frac{-\beta}{M-1}}.
$$
\end{enumerate}
\end{lemma}

The proof of Lemma \ref{theo3lemm2} is presented in Appendix \ref{appendixB}. It is technical and the strategy of the proof is similar to the proof of Proposition 8 in \cite{galves2024fastconsensusmetastabilityhighly}. 

\begin{corollary}\label{theo3coro2}
    For any $\beta \geq 0$, 
    $$
       c_\beta \geq \frac{1}{2}N^{-3}(M+1)^{-3}e^{\frac{\beta}{M-1}}.
    $$
\end{corollary}

\begin{proof}
By the previous lemma,  we have that
$$
e^{-1}= \mathbb{P}(R^{\beta,l}(\mathcal{C}^{-o})>c_{\beta}) \geq e^{-2c_{\beta}N^3(M+1)^3e^{\frac{-\beta}{M-1}}}.
$$
We finish the proof by rearranging the terms.
\end{proof}

We are now ready to prove Theorem \ref{Theorem 3}.

\begin{proof}
It is sufficient to successively prove that all necessary conditions for using Proposition \ref{prop:evamonm} hold. 

For assumption \eqref{eq:propevamonm_assump1} we simply apply part 1 of Lemma \ref{theo3lemm2} with $t=1$ to obtain  $s_1=1$ and $\varepsilon_1=2N^3(M+1)^3e^{\frac{-\beta}{M-1}}$.

For assumption \eqref{eq:propevamonm_assump2}, we simply apply Lemma \ref{theo3lemm1} for $s_2=2\beta$ and $\varepsilon_2=(M+1)^2N^2e^{-\beta /(M+1)N}$. Note that $\varepsilon_1+\varepsilon_2<1/2$ for $\beta$ sufficiently big.

To prove that \eqref{eq:propevamonm_assump3} holds, note that $s_2/c_{\beta}$ is bounded from above by
\begin{equation} \label{eq:boundbysup}
2\beta e^\frac{-\beta}{2(M-1)}2(M+1)^3N^3e^\frac{-\beta}{2(M-1)} \leq 8e^{-1}(M-1)(M+1)^3N^3e^\frac{-\beta}{2(M-1)}.
\end{equation}
The inequality above follows from $\sup_{\beta\geq 0}(\beta e^{-\beta /[2(M-1)]})=2e^{-1}(M-1)$. This implies that
$$
\frac{s_2}{c_{\beta}}\vee \varepsilon_2 \leq 8e^{-1}(M+1)^4N^3e^{-\beta/(M+1)N}.
$$
We conclude that \eqref{eq:propevamonm_assump3} holds with $\delta=1/(M+1)N$ and $C=8e^{-1}(M+1)^4N^3$. 

Finally, by part 2 of Lemma \ref{theo3lemm2}, by using the same bound we obtained in \eqref{eq:boundbysup}, we have that \eqref{eq:propevamonm_assump4} holds with $\theta=1/(2(M-1))$ and $K=8e^{-1}N^3(M+1)^4$, which concludes the proof. 
\end{proof}

\subsection{Proof of Theorem \ref{Theorem 4}} \label{subsec:T4}

In the following, we deal with the model introduced in Section \ref{sec:communicationbias}. For the purpose of simplifying the notation, we will omit the parameter $\beta$, as it is not relevant to the following results, which hold for any $\beta>0$. Thus, we shall write $(U_{t}^{\alpha,u})_{t\geq 0}$ instead of $(U_{t}^{\alpha, \beta, u})_{t\geq 0}$ to denote the time evolution of the social pressure on the model of biased communication induced by the generator \eqref{biasGenerator}. 
Similarly, we will use the notation $(T_n,A_n,O_n)_n$ instead of $(T_n^{\alpha,\beta,u},A_n^{\alpha,\beta,u},O_n^{\alpha,\beta,u})_n$ to respectively represent the successive expression times, actors and opinions driving the evolution of the random process. 

To prove Part 1 of Theorem \ref{Theorem 4}, let us first consider the following theorem, which similarly to part (i) of Theorem \ref{Theorem 1} states that the random process is defined for any positive time $t$.

\begin{theorem}\label{Theorem 1_alpha}
    For any $\beta\geq 0$, $\alpha<0$ and any starting matrix $u\in\mathcal{S}^{\alpha}$, the sequence $(T_m:m\geq 1)$ of jumping times of the process $(U^{\alpha, u}_t)_t$ satisfies
    \[\mathbb{P}(\sup\{T_m:m\geq 1\} = \infty)=1. \]
\end{theorem}

\begin{proof}
First, note that we can prove exactly as Proposition \ref{prop1} that for any starting matrix $u\in\st^{\alpha}$,
    $$
        \inf\{n\geq 1:\max_{o\in \mathcal{O}} U^{\alpha,u}_{T_{n-1}}(A_n,o)<N \} \leq N.
    $$
From this point on, the proof of Theorem \ref{Theorem 1_alpha} follows exactly as the proof of Theorem \ref{Theorem 1}.
\end{proof}

Now, we will prove two auxiliary results.
Let us consider the following notation
$$
B_N^{\alpha}=\{u \in \st^{\alpha}: \max\{u(a,o): a \in \A,o\in \opn\}\leq N\}.
$$
Moreover, for any $\beta\geq0, u\in\mathcal{S}^{\alpha}, \theta \subset\mathcal{S}^{\alpha}$, we define

$$
    R^{\alpha,u}(\theta):=\inf\{t\geq0:U^{\alpha,u}_t\in\theta\}.
$$

\begin{proposition} \label{prop1_teo4} For any $\alpha <0$, $\beta \geq 0$ and
for any $u \in \mathcal{S}^\alpha$,
$$
\P\left(U_{T_N}^{\alpha,u} \in B_N^{\alpha}\right)\geq  (NM)^{-N}.
$$
\end{proposition}
 
\begin{proof}
Let
$$
    \xi_j^{\alpha,u}=\{(A_j,O_j) \in \text{argmax}\{U_{T_{j-1}}^{\alpha, u}(a,o):(a,o)\in \A\times \opn\}\}.
$$
By following the same ideas of the proof of Proposition \ref{prop2}, we have that for $u \in \st^{\alpha}$,
$$
\bigcap_{j=1}^{N}\xi_j^{\alpha,u}\subset \{U_{T_N}^{\alpha,u} \in B_N^{\alpha}\}.
$$
Since $\xi_j^{\alpha,u}$ is the most probable choice of actor/opinion between the $NM$ possible choices, we have that for any $u\in \st^{\alpha}$,
$$
\P(\xi_1^{\alpha,u})\geq (MN)^{-1}.
$$
By iterating the inequality, we conclude the proof.
\end{proof}

\begin{proposition} \label{prop2_teo4} Let us write $(A_n^u)_n$ to denote the sequence of actors expressing opinions in the system with biased communication with initial matrix $u\in \st^{\alpha}$. For any $\alpha <0$ and $\beta >0$, it follows that
$$
\inf_{u \in B_N^{\alpha}}\P\left( \bigcap_{j=1}^{+\infty}\{A^u_j=A^u_1\}\right)>0.
$$
\end{proposition}
\begin{proof}
Without loss of generality, 
let us consider $u \in B_N^{\alpha}$ such that $u(1,o)=0$ for all $o \in \opn$. For any $n\geq 1$, we have that
\begin{multline} \label{eq:teo4}
\P\left( \bigcap_{m=1}^{n}\{ A_m = 1\}\right)=\sum_{\{o_j\}_j \in \opn^n}\P\left( \bigcap_{m=1}^{n}\left\{ A_m = 1, O_m=o_m\right\}\right)= \\
\sum_{\{o_j\}_j \in \opn^n}\P(A_1=1,O_1=o_1)
\prod_{m=2}^n\P\left(A_m=1,O_m=o_m\ \Big| \bigcap_{k=1}^{m-1} \{A_k=1,O_k=o_k\} \right).
\end{multline}

Denote
$$
\lambda_1=(N-1)Me^{\beta N}.
$$
Note that $(M+\lambda_1)^{-1}$ is a lower bound for $\P(A_1=1,O_1=o_1)$ for any initial matrix $u \in B_N^{\alpha}$ such that $u(1,o)=0$, for all $o \in \opn$. For $m\geq 2$, denote
\begin{equation} \label{def:lambdam}
\lambda_m=\lambda_m(o_1,\ldots,o_{m-1})=(N-1)\sum_{o\in \opn }e^{\beta N+\beta\sum_{k=1}^{m-1}(\mathbf{1}\{o_k=o\}-\gamma\mathbf{1}\{o_k\neq o\})}.
\end{equation}
Note that $(M+\lambda_m)^{-1}$ is a lower bound for $\P(A_m=1,O_m=o_m)$ for any matrix obtained after actor $1$  successively expresses the opinions $o_1, \ldots, o_{m-1}$ when the initial matrix is $u \in B_N^{\alpha}$ such that $u(1,o)=0$, for all $o \in \opn$. 

With this notation, the right-hand side of \eqref{eq:teo4} is lower bounded by
\begin{multline*}
    \sum_{\{o_j\}_j \in \opn^n}
\prod_{m=1}^n\frac{1}{M+\lambda_m} = \sum_{\{o_j\}_j \in \opn^n} \frac{1}{M^n}e^{
\sum_{m=1}^n\ln\left(\frac{M}{M+\lambda_m}\right)}=
\\ \sum_{\{o_j\}_j \in \opn^n} \frac{1}{M^n}e^{
\sum_{m=1}^n\ln\left(1-\frac{\lambda_m}{M+\lambda_m}\right)}.
\end{multline*}
By using the inequality $\ln(1+x)\geq x/(1+x)$, for $x>-1$, we have that
the right-hand side of equation above is lower bounded by
$$
\sum_{\{o_j\}_j \in \opn^n} \frac{1}{M^n}e^{
-\sum_{m=1}^n\left(\frac{\lambda_m}{M+\lambda_m}\right)\Big/\left(1-\frac{\lambda_m}{M+\lambda_m}\right)} \geq \sum_{\{o_j\}_j \in \opn^n} \frac{1}{M^n}e^{
-\sum_{m=1}^n\lambda_m}=
$$
$$
\mathbb{E}_{V_1,\ldots, V_n}\left(e^{-\sum_{m=1}^{n}\lambda_m(V_1,\dots,V_{m-1})} \right),
$$
where $V_1,V_2 \ldots$ are i.i.d. uniform random variables in $\{1,\ldots, M\}$. 
We conclude that, for any $n\geq 1$,
$$
\P\left( \bigcap_{m=1}^{n}\{ A_m = 1\}\right) \geq \mathbb{E}_{V_1,V_2,\ldots}\left(e^{-\sum_{m=1}^{\infty}\lambda_m(V_1,\dots,V_{m-1})} \right).
$$
In the following, we will omit $V_1,V_2\ldots,$ in the notation of the expectation $\mathbb{E}_{V_1,V_2,\ldots}$.

Recalling \eqref{def:lambdam}, we have that the term above is equal to
\begin{equation} \label{eq:expectation}
\mathbb{E}\left[\exp\left(-(N-1)\sum_{m=1}^{+\infty}\sum_{o\in \opn}e^{\beta [N+\sum_{k=1}^{m-1}(\mathbf{1}\{V_k=o\}-\gamma\sum_{\tilde{o}\neq o}\mathbf{1}\{V_k= \tilde{o}\})]} \right)\right].
\end{equation}

Consider for $k\in \{1,2,\ldots\}$  and for $\epsilon>0$ the event
$$
E_{\epsilon}^{k}=\left\{\left(\frac{1}{M}-\epsilon\right)n \leq \sum_{m=1}^{n}\mathbf{1}\{V_m=o\}\leq \left(\frac{1}{M}+\epsilon\right)n, n\geq k, o\in \opn\right\}.
$$
Since for any $o\in \opn$ we have
$$
\frac{\sum_{m=1}^{n}\mathbf{1}\{V_m=o\}}{n}\to \frac{1}{M}, \text{ a.s. as } n\to \infty,
$$
we conclude that for any $\epsilon>0$ there exists $k=k(\epsilon)\geq 1$ such that
$
\P(E_{\epsilon}^{k})>0.
$

In the following, let us consider a generic constant $C=C(M,N,\beta,\epsilon)>0$.
By the law of total expectations, for any $\epsilon>0$ and $k\geq 1$ such that $\P(E_{\epsilon}^{k})>0$, we have that \eqref{eq:expectation} is lower bounded by
\begin{multline}
\mathbb{E}\left[\exp\left(-C\sum_{m=1}^{+\infty}\sum_{o\in \opn}e^{\beta [\sum_{k=1}^{m-1}(\mathbf{1}\{V_k=o\}-\gamma\sum_{\tilde{o}\neq o}\mathbf{1}\{V_k= \tilde{o}\})]} \right)\Big|E_{\epsilon}^{k}\right]\P(E_{\epsilon}^{k})  \\
\label{eq:prooft4end}
\geq C\exp\left(-C\sum_{m=1}^{+\infty}e^{\beta [(\frac{1}{M}+\epsilon)m-\gamma(M-1)(\frac{1}{M}-\epsilon)m]} \right) \\ = C\exp\left(-C\sum_{m=1}^{+\infty}e^{-\beta c(\alpha,\epsilon)m} \right),
\end{multline}
where
$$
c(\alpha,\epsilon):=-\left(2\epsilon +\alpha(M-1)(\frac{1}{M}-\epsilon)\right).
$$
Note that for any $\alpha<0$, there exists $\epsilon>0$ such that $c(\alpha,\epsilon)>0$.
With such a choice of $\epsilon$, we conclude that
the term in the right-hand side of \eqref{eq:prooft4end} is positive,
which concludes the proof.
\end{proof}

Now let us prove Part $1$ of Theorem \ref{Theorem 4}.

\begin{proof}
By Proposition \ref{prop2_teo4} and the strong Markov property at time $T_N$, there exists $c=c(\alpha,M,N,\beta)>0$ such that for any $u\in \st^{\alpha}$,
\begin{multline*}
\P\left(\bigcap_{m=N+1}^{\infty}\{A_m^{u}=A_{N+1}^{u}\}\right) \geq
\mathbb{E}\left[\mathbf{1}\{U_{T_N}^{\alpha,u} \in B_N^{\alpha}\}\,
\P^{U_{T_N}^{\alpha,u}}\left(\bigcap_{m=1}^{\infty}\{A_m=A_1\}\right)\right]\\
\geq (NM)^{-N}\inf_{v \in B_N^{\alpha}}\P\left(\bigcap_{m=1}^{\infty}\{A^v_m=A^v_1\}\right)=:c>0,
\end{multline*}
where the last inequality uses Proposition \ref{prop1_teo4}. We now define the successive failure times of the event above. Set
$\eta_0 := 0$ and, recursively, for any $n \geq 0$,
\[
\eta_{n+1} := \inf\big\{ m > \eta_n + N + 1 :\ A^{u}_{m} \neq A^{u}_{\eta_n + N + 1} \big\},
\qquad \inf \emptyset := +\infty .
\]
In words, after each failure the process is given $N$ jumps to reach $B^{\alpha}_{N}$, and $\eta_{n+1}$ is the first subsequent expression performed by a different actor than the one expressing at step $\eta_n + N + 1$. Note that $\{\eta_{n+1} < \infty\} \subset \{\eta_n < \infty\}$ and that, on $\{\eta_n < \infty\}$, $T_{\eta_n}$ is a stopping time with $U^{\alpha,u}_{T_{\eta_n}} \in \mathcal{S}^{\alpha}$. Each event $\{\eta_n < \infty\}$ is measurable with respect to $\mathcal{F}_{T_{\eta_n}}$, and by the strong Markov property applied at time $T_{\eta_n}$ together with the uniform bound above,
\[
\mathbb{P}(\eta_{n+1} = \infty \,|\, \mathcal{F}_{T_{\eta_n}}) \geq c
\quad \text{on } \{\eta_n < \infty\},
\]
so that $\mathbb{P}(\eta_{n+1} < \infty) \leq (1-c)\, \mathbb{P}(\eta_n < \infty)$. Therefore
$$
\P\left(\bigcap_{n\geq 1}\{\eta_n<\infty\}\right)=\lim_{n\to\infty}\P(\eta_n<\infty)=0,
$$
that is, almost surely $\eta_n=\infty$ for some $n$, which means that from some time on a single actor expresses all the opinions in the system. This concludes the proof of Part 1 of Theorem \ref{Theorem 4}.
\end{proof}

To prove Part 2 of Theorem \ref{Theorem 4}, we show how the proofs presented in Sections \ref{subsec:T1}, \ref{subsec:T2} and \ref{subsec:T3}  to deal with the model without bias communication can be modified. This proof is deferred to Appendix \ref{ap:part2theorem4}.

\section{Discussion and Perspectives} \label{sec:discussion}

We have extended the model introduced by \cite{galves2024fastconsensusmetastabilityhighly} to the multiple-opinion case, generalizing and proving the main theorems presented in \cite{galves2024fastconsensusmetastabilityhighly}. Beyond this endeavor, we revealed a novel phase transition induced by negative communication bias. This framework can be a starting point for further investigations of stochastic models of interacting agents.

The communication operator $\pi$ was originally designed in order to have a neutral impact on the system. This ensured that our generalized model faithfully reproduced the original model for $M=2$. The natural question we asked ourselves after this observation is: What happens if we disturb this equilibrium? This led to the discovery of the phase transition phenomenon and its subsequent proof. Crucially, this theorem can only be stated in our new formulation of the model. This raises the question of whether the new formulation not just allows for an arbitrary number of opinions but also provides a better framework to think of the model and upon which to create new generalizations or extensions.

\subsection{Negative communication bias: posting zero and information overload}

Let us consider a version of the model introduced here in which only the actor $1$ has a negative communication bias, meaning that the opinions that this actor expresses modify the matrix of social pressures according to the map \eqref{biasmap} with $\alpha<0$, while the actors $\{2,\ldots, N\}$ do not have communication bias, meaning that the opinions that these actors express modify the matrix of social pressures according to the map \eqref{modelmap}.

The Propositions \ref{prop1_teo4} and \ref{prop2_teo4} can be easily adapted for this model, and as a consequence, a version of part 1 of Theorem \ref{Theorem 4} holds for this model: for any $\beta>0$ and for any $\alpha<0$, we have for any initial matrix of social pressures that with probability $1$ there exists a finite time in which all actors, except actor $1$, stop expressing opinions.

From a modeling point of view, this means that in this network, if the opinions expressed by an actor strongly discourage the other actors in the network from expressing different opinions, then eventually, the other actors in the network stop expressing opinions.
If this actor is replaced by influencers or robots that act as actors in the network and express each one of the opinions with the same rate, this will lead to a situation in which all actors on the network stop expressing opinions.

In the model \eqref{biasGenerator} with negative communication bias and in the model described above, the majority or all actors eventually stop expressing opinions. The phenomenon of fewer people posting on social networks was recently observed empirically (see \cite{lesspostingonsocialnetworks}). This behavior of the model is reminiscent of what \cite{Chayka} called \textit{Posting Zero}, described as ``a point at which normal people --- the unprofessionalized, uncommodified, unrefined masses --- stop sharing things on social media as they tire of the noise, the friction, and the exposure''. 

Moreover, this behavior is reminiscent of one of the effects of information overload. Among the consequences of information overload, \cite{informationoverload} states that the ``introduction of excessive information can lead to potential paralysis and delays in decision-making''.

\subsection{Empirical validation}

A well-grounded critique pertaining to our work could be formulated as follows: the microscopic model of social interaction defining the social dynamics is too simplified. Indeed, the best justification to model the behavior of an individual as we have done would be to explicitly show how it yields a satisfactory description of observed macroscopic phenomena. However, empirical validation remains an open challenge. 

The real-world problem motivating the creation of the model in \cite{galves2024fastconsensusmetastabilityhighly} was to obtain a description of how WhatsApp was an instrument for far-right groups meddling in the Brazilian 2018 presidential election. However, direct validation in the context of large-scale social communication (e.g., messaging platforms) is challenging due to incomplete network data, privacy constraints, and the difficulty of quantifying semantic content.

We hypothesize these problems could be circumvented provided we change our focus to other social phenomena. One should then look for dynamics created by interconnected actors taking decisions, where the macroscopic equilibrium is mostly stable, but punctuated by occasional, hard-to-predict shifts:

\begin{enumerate}
    \item \textbf{Jurisprudence}: Recent empirical studies show that jurisprudence exhibits path-dependent stability driven by precedent, but also episodic and sometimes unpredictable shifts due to judicial discretion, ideological change, and institutional dynamics. Indeed, \cite{mones2021network_judicial} and \cite{lee2025commonlawagescenturies} study the judicial social network to uncover the mechanisms of jurisprudence in Europe and in the United States.
    \item \textbf{Market cycles}: Empirical research shows that financial markets exhibit persistent cyclical dynamics and partial predictability, yet remain intrinsically nonstationary and prone to abrupt regime shifts, driven by endogenous instability and evolving structural mechanisms. This is classically described in \cite{minsky1986stabilizing}, but this active field of economic research has seen novel contributions in \cite{ninomiya2022instability} and \cite{akaev2021minsky}.
\end{enumerate}

Should one of these leads prove fruitful, one should remain mindful that our model is only trying to describe the mechanisms behind the consensus dynamics occurring in a given social setting. A full description of the underlying social phenomenon would have to include other factors which depend on the specific context at hand.

\subsection{The path ahead}

The results established in this article open several directions for future work. Indeed, one may consider generalizations along many different axes: sparser network topologies, the introduction of homophily among actors, mean field limits as the number of actors grows, and statistical inference questions such as model selection and interaction graph estimation.

First, one may consider sparser topologies: in an Erd\H{o}s--R\'enyi graph $G(N,p)$, the operator $\pi^{a,o}$ defined by \eqref{modelmap} would only affect the neighbors of $a$, and a natural question is whether Theorems~2 and~3 persist near the connectivity threshold $p \sim \log(N)/N$, below which the graph is no longer connected with high probability. 

Second, one may introduce homophily, that is, the tendency of actors to be more receptive to opinions expressed by actors with similar social pressures. This could be encoded either in the expression rates, by weighting the influence of actor $b$ on actor $a$ according to a similarity measure between their social pressure vectors, or in the communication bias parameter $\alpha$, by allowing it to depend on the pair $(a,b)$ and to be negative when the two actors hold dissimilar opinions. Both directions raise the question of whether consensus formation and metastability are robust features of the model, or whether they depend critically on the symmetry assumptions of the present framework.

A mean field version of the model introduced by \cite{galves2024fastconsensusmetastabilityhighly} was considered by \cite{evalaxa}. In this framework, the limit behavior of the system was studied as the number of actors in the network diverges, establishing the propagation of chaos property. The limit system exhibits a phase transition in which the number of invariant probability measures of the limit system changes according to a parameter that describes how much the social pressure of an actor changes when one of the other actors expresses an opinion.
A mean field version of the model with multiple opinions introduced here can also be considered, and we conjecture that its behavior is similar.

Some of the statistical inference questions considered for the Galves-Löcherbach model can be adapted to the framework of this social network model. The results presented in \cite{ostduarte3} and \cite{desantis} consider procedures aimed to estimate the graph of interactions between neurons. Such procedures can identify whether the synaptic weights are positive, negative, or null. However, it remains an open question to verify whether it is possible to adapt the results presented in \cite{ostduarte3} and \cite{desantis} to the social network model introduced here.

\newpage
\appendix

\section*{Appendix}

\section{Proof of Proposition \ref{prop3}}\label{appendix1}

In order to prove Proposition \ref{prop3}, we shall first prove Lemmas \ref{appendix_lemma1} and \ref{appendix_lemma2}. For that, we need the following definitions.

\begin{definition}
 For any $o\in\mathcal{O}$, we call $\st^o$ the set of matrices $u\in\st$ such that there exists $r\in [0,1), n(u)\in\{1,...,N-1\}$ and a sequence of $n(u)$ different actors $a_1(u),...,a_{n(u)}(u)\in\mathcal{A}$ satisfying
        $$
            u(a_j(u),o)\geq j+r, \forall j \in\{1,...,n(u)\},\text{ and}
        $$
        $$
            u(a,p)\leq n(u)+r-\frac{1}{M-1}, \forall a\in\mathcal{A},p\in\mathcal{O}\setminus\{o\}.
        $$    
Moreover, for any initial matrix $u\in\st$, let
        $$
            \tau(u):=\inf\{n\geq 1, A_n\in\{A_1,...,A_{n-1}\}\}.
        $$
        This is the number of social expressions until an actor expresses an opinion for the second time. Note that $\tau(u)\in\{2, ..., N+1\}$.
\end{definition}

Recall that the skeleton process is given by $\Tilde{U}_n^{\beta, u}=U_{T_n}^{\beta, u}$, for any $n\geq 0$.

\begin{lemma}\label{appendix_lemma1}

For any initial matrix $u\in\st$, the event $\xi_{\tau(u)}^u$ implies that 
$$
    \Tilde{U}^{\beta,u}_{\tau(u)}\in\bigcup_{o\in\mathcal{O}}\st^o.
$$
    
\end{lemma}

\begin{proof}
    We begin by pointing out that $\{\tau(u)=2\}\cap\xi_2^u$ is only possible if $\Tilde{U}^{\beta,u}_1=0$. If this occurs, the result follows trivially. 
    Thus, we may suppose that $\tau(u)\geq 3$ in what follows. Let us pay attention to what happens at the moment $\tau(u)-1$. We have $$
\Tilde{U}^{\beta,u}_{\tau(u)-1}(A_{\tau(u)-1,\cdot})=0
    $$
    and for any $j\in\{2,...,\tau(u)-1\}$,
    \begin{equation} \label{eq:dif}
\|\Tilde{U}^{\beta,u}_{\tau(u)-1}(A_{j}, .)-\Tilde{U}^{\beta,u}_{\tau(u)-1}(A_{j-1},.)\|_\infty=1.        
    \end{equation}
    Assuming $\xi_{\tau(u)}^u$, we know that 
    $$
\left(A_{\tau(u)},O_{\tau(u)}\right)\in\arg\max_{(a,o)\in\mathcal{A}\times\mathcal{O}}\Tilde{U}^{\beta,u}_{\tau(u)-1}(a,o).
    $$
    Then, let
    $$
        m=\Tilde{U}^{\beta,u}_{\tau(u)-1}(A_{\tau(u)},O_{\tau(u)})\geq \lfloor m \rfloor.
    $$
    
    Therefore by \eqref{eq:dif}, we obtain a sequence of $\lfloor m \rfloor+1$ distinct actors
    $$
        \{a_0,...,a_{\lfloor m \rfloor}\}\subset\{A_1,...,A_{\tau(u)-1}\},
    $$
    with $a_{\lfloor m \rfloor}=A_{\tau(u)} $, such that for any $j\in\{0,...,\lfloor m \rfloor\}$,
    $$
        \Tilde{U}^{\beta,u}_{\tau(u)-1}(a_j,O_{\tau(u)})\geq j+(m-\lfloor m \rfloor).
    $$
This implies that at step $\tau(u)$, we have $\lfloor m \rfloor+1$ distinct actors $a_0',...,a_{\lfloor m \rfloor}'\in\mathcal{A}$ such that for any $j\in\{1,...,\lfloor m \rfloor\}$,
    
    $$
        \Tilde{U}^{\beta,u}_{\tau(u)}(a_j',O_{\tau(u)})\geq j+(m-\lfloor m \rfloor).
    $$
    and $\forall a\in\mathcal{A},p\in\mathcal{O}\setminus\{O_{\tau(u)}\}$ we have
    $$
        \Tilde{U}^{\beta,u}_{\tau(u)}(a,p)\leq m+(m-\lfloor m \rfloor)-\frac{1}{M-1},
    $$
    which concludes the proof taking $n(u) =\lfloor m \rfloor, r = m-\lfloor m \rfloor$.
\end{proof}
 
\begin{lemma}\label{appendix_lemma2}

For any matrix $u\in\st^o$, the event $\bigcap^{(M-1)(n(u) + 1)-1}_{j=1}\xi_j^u$ implies that
$$
    \Tilde{U}^{\beta,u}_{(M-1)(n(u) + 1)-1}\in\mathcal{C}^o.
$$    
\end{lemma}

\begin{proof}
    Let $o \in \mathcal{O}$ and $u \in \mathcal{S}^o$. If the event $\xi_1^u$ occurs, then $\Tilde{U}^{\beta,u}_{0}(A_1,O_1)=\max_{(a,o)\in \A\times \opn} u(a,o)$ and $ O_1=o$. Thus, there are $n(u)$ different actors such that for any $j\in\{0,..., n(u)-1\}$, there is an $a_j(u)\in\mathcal{A}$ where
    $$
        \Tilde{U}^{\beta,u}_{0}(a_j(u),o)=u(a_j(u),o)\geq j + r,
    $$
    and therefore
    $$
        \Tilde{U}^{\beta,u}_{1}(a_j(u),o)\geq j+r+1.
    $$
    Moreover,
    $$
        \forall a\in\mathcal{A}, p\in\mathcal{O}\setminus\{o\}, \Tilde{U}^{\beta,u}_{1}(a,p)\leq n(u) + r -\frac{2}{M-1}.
    $$
    More generally, for any $k\in\mathbb{N}$, if the event
    $$
        \bigcap^k_{j=1}\xi_j^u
    $$
    occurs, then there are $n(u)$ different actors $a_1(\Tilde{U}^{\beta,u}_{k-1}),...,a_{n(u) }(\Tilde{U}^{\beta,u}_{k-1})\in\mathcal{A}$ such that for any $j\in\{1,..., n(u) \}$
    $$
        \Tilde{U}^{\beta,u}_{k}(a_j(\Tilde{U}^{\beta,u}_{k-1}),o)\geq j + r,
    $$
    $$
        \forall a\in\mathcal{A}, p\in\mathcal{O}\setminus\{o\}, \Tilde{U}^{\beta,u}_{k}(a,p)\leq n(u) + r-\frac{k+1}{M-1}.
    $$
  By taking $k=(M-1)(n(u)+ 1)-1$ we conclude the proof.
\end{proof}

Having achieved these results, we may now accomplish the proof of Proposition \ref{prop3}.

\begin{proof}
    Let $u\in\st$. To begin this proof, we first observe that $\tau(u)\leq N+1$, so by Lemma \ref{appendix_lemma1} we have
    $$
        \left\{\bigcap^{N+1}_{j=1}\xi_j^u \right\}\subset \left\{\Tilde{U}^{\beta,u}_{N+1}\in\bigcup_{o\in\mathcal{O}}\st^o \right\}.
    $$
    Now, observe that $n(u)\leq N-1$, so by Lemma \ref{appendix_lemma2} we obtain
    $$
        \left\{\bigcap^{MN}_{j=1}\xi_j^u \right\}\subset \left\{\Tilde{U}^{\beta,u}_{MN}\in\bigcup_{o\in\mathcal{O}}\mathcal{C}^o \right\}.
    $$
    To conclude, we consider that by definition, if $v\in\mathcal{C}^o$ and event $\bigcap^{N}_{j=1}\xi_j^u$ occurs, then $\Tilde{U}^{\beta,u}_{N}\in\mathcal{L}^o$.
\end{proof}

\section{Proof of Lemma \ref{theo3lemm2}}\label{appendixB}

\begin{proof} 
The idea of the proof of part $1$ is informally described as follows. Starting from $\mathcal{L}^o$, a sequence of expressed opinions equal to $o$ keeps the system in $\hat{\mathcal{L}}^o$ (recall Definition \ref{def:hatladder} and Remark \ref{rem:hatladder}). From $\hat{\mathcal{L}}^o$, an opinion different from $o$, expressed by the unique actor with null social pressure for all opinions, may make the process exit $\hat{\mathcal{L}}^o$. However, this keeps the process at the extended consensus set, given by
$$
\hat{\mathcal{C}}^{o}:=\left\{u\in\mathcal{S}: \max_{b\in \A}u(b,o)\geq 1, \min_{b\in \A}u(b,o)= 0,\max_{b\in \A, p\neq o}u(b,p)<1\right\}.
$$
Note that for any $v\in \hat{\mathcal{C}}^{o}$, $\xi_1^{v}$ implies that $O_1=o$ and $\tilde{U}_{1}^{\beta,v}\in \hat{\mathcal{C}}^{o}$. By Proposition \ref{prop3}, this implies that if the event $\bigcap^{(M+1)N}_{j=1}\xi_j^{v}$ occurs, then $\tilde{U}^{\beta,v}_{(M+1)N}\in\mathcal{L}^o$. 

Thus, consider the following mechanism.
\begin{enumerate}
    \item Starting from $\mathcal{L}^o$, the process does not exit $\hat{\mathcal{L}}^o$ until an opinion different from $o$ is expressed in the system. This opinion is expressed by the unique actor with null social pressure for all the opinions, leading the process to $\hat{\mathcal{C}}^{o}$.
    \item From $\hat{\mathcal{C}}^{o}$, a sequence of events $\bigcap^{(M+1)N}_{j=1}\xi_j^{v}$ leads the process to $\mathcal{L}^{o}$ again.
\end{enumerate}
While this mechanism is repeated, the process does not visit $\mathcal{C}^{-o}$. This gives a lower bound for the hitting time of the set $\mathcal{C}^{-o}$. In the following, let us formalize the description above.

For any fixed $o\in\mathcal{O},l \in\mathcal{L}^o$, let
    $$
        \tau^{-o}_1:=\inf\{T_n:U^{\beta,l}_{T_{n-1}}(A_n,O_n)<0\}.
    $$
    This is the first time an actor expresses an opinion for which it has negative social pressure. By denoting $n^{-o}_0:=0$ and for $j\geq1$, $n^{-o}_j:=\min\{n>n^{-o}_{j-1}, O_n\neq o\},$ consider also 
    $$
    \tau^{-o}_2:=\inf\left\{T_{n_j^{-o}}:\bigcap_{k=1}^{(M+1)N}\xi_{n_{j}^{-o}+k}^{l} \text{ does not occur}\right\}.
    $$
    For any $l\in \mathcal{L}^o$, we have
    $$
        R^{\beta,l}(\mathcal{C}^{-o})\geq R^{\beta,l}(\mathcal{S}\setminus(\hat{\mathcal{C}}^o\cup \mathcal{C}^o))\geq \min\{\tau^{-o}_1,\tau^{-o}_2\}.
    $$

    To bound $\tau^{-o}_1$, note that the rate at which the process has an actor with strictly negative social pressure expressing an opinion is bounded above by $NMe^{\frac{-\beta}{M-1}}$. 
    This implies that for any $t>0$,
$$
        \mathbb{P}(\tau^{-o}_1\geq t)\geq \mathbb{P}(Exp(NMe^{\frac{-\beta}{M-1}})\geq t),
    $$
where $Exp(NMe^{\frac{-\beta}{M-1}})$ is an exponentially distributed random variable with mean $e^{\frac{\beta}{M-1}}/NM$. 
    
Now, for $\tau^{-o}_2$, we have that
$$
\tau^{-o}_2=\sum^J_{j=1}T_{n^{-o}_j}-T_{n^{-o}_{j-1}},
$$
where $J:=\inf\{j\geq 1: \bigcap_{k=1}^{(M+1)N}\xi_{n_{j}^{-o}+k}^{l} \text{ does not occur}\}$. Note that for any $j=2,\ldots,J$, we have
$$
T_{n^{-o}_j}-T_{n^{-o}_{j-1}}\geq T_{n^{-o}_j}-T_{n^{-o}_{j-1}+(M+1)N}.
$$
Moreover for any $j= 2, \ldots, J$, in the event $\tau^{-o}_1\geq T_{n^{-o}_j}$ we have that $U_{T_{n^{-o}_{j-1}+(M+1)N}}^{\beta, l} \in \mathcal{L}^o$. Finally, recall that $U_0^{\beta,l} \in  \mathcal{L}^o$.
For any element of $\hat{\mathcal{L}}^o$, the rate at which the process has an opinion expression for $p\neq o$ is bounded above by $NM$. This implies that, for any $j =1,\ldots,J$, for any $t > 0$,
$$
        \mathbb{P}(T_{n^{-o}_j}-T_{n^{-o}_{j-1}} >t|\tau^{-o}_1\geq T_{n^{-o}_j})\geq\mathbb{P}\left(E_j\geq t \right),
$$
where $(E_j)_j$ is a sequence of i.i.d. exponentially distributed random variables with mean $1/(MN)$. Moreover, by Proposition \ref{prop4} and Remark \ref{remark:bernoulli}, we have that for any $u\in \st$,
$$
\P\left(\left\{\bigcap_{j=1}^{(M+1)N}\xi^{l}_j\right\}^c\right)\leq1-\left(\frac{e^{\frac{\beta}{M-1}}}{e^{\frac{\beta}{M-1}}+MN}\right)^{(M+1)N}
$$
$$
\leq  (M+1)MN^2e^{\frac{-\beta}{M-1}}
$$
Therefore, for any $t>0$,
$$
        \mathbb{P}(\tau^{-o}_2\geq t|\tau^{-o}_1\geq t)\geq\mathbb{P}\left(\sum^{G}_{j=1}E_j\geq t \right),
    $$
where $G$ is a random variable independent from $(E_j)_{j\geq1}$ with Geometric distribution assuming values $\{1,2,...\}$ with parameter $p=\min(1,(M+1)MN^2e^{\frac{-\beta}{M-1}})$. This implies that for any $t>0$,
\begin{multline*}
\mathbb{P}(\tau^{-o}_2\geq t|\tau^{-o}_1\geq t)\geq 
          \mathbb{P}(Exp(MNp)\geq t) \geq \\
        \mathbb{P}(Exp((M+1)M^2N^3e^{\frac{-\beta}{M-1}})\geq t),
\end{multline*}
    where $Exp((M+1)M^2N^3e^{\frac{-\beta}{M-1}})$ is an exponentially distributed random variable with mean $[(M+1)M^2N^3e^{\frac{-\beta}{M-1}}]^{-1}$. Therefore, for any $t>0$,
\begin{multline*}
\mathbb{P}(R^{\beta,l}(\mathcal{S}\setminus(\hat{\mathcal{C}}^o\cup \mathcal{C}^o))\geq t)\geq\\\mathbb{P}(Exp(NMe^{\frac{-\beta}{M-1}})\geq t)\mathbb{P}(Exp((M+1)M^2N^3e^{\frac{-\beta}{M-1}})\geq t).
\end{multline*}
By rearranging the terms of the inequality above we conclude the proof of part $1$.

To prove part 2 of Lemma \ref{theo3lemm2}, just note that for any $u \in \mathcal{C}^o$ and for $l\in \mathcal{L}^o$,
$$
\P(R^{\beta,u}(\mathcal{C}^{-o}) \leq t) \leq 1-
\P\left(\bigcap_{j=1}^{N}\xi^{u}_j\right)\mathbb{P}(R^{\beta,l}(\mathcal{C}^{-o})>t).
$$
We finish the proof by considering Proposition \ref{prop4} and Remark \ref{remark:bernoulli} and noting that $1-e^{-x}\leq x$ for all $x \geq 0$.
\end{proof}

\section{Proof of Part 2 of Theorem \ref{Theorem 4}} \label{ap:part2theorem4}
Note that the case $\alpha=0$ is exactly the model without bias communication defined by \eqref{generator:withoutbias}. Therefore, in the following, let us consider $0 < \alpha < \frac{1}{M-1}$.

First, note that we can prove the following proposition exactly as we did to prove Proposition \ref{prop1}.
\begin{proposition}\label{prop1_2}
    For any starting matrix $u\in\st^{\alpha}$, we have:
    $$
        \inf\{n\geq 1: \|U^{\alpha,u}_{T_{n-1}}(A_n^{\alpha},\cdot)\|_\infty<N \} \leq N.
    $$
\end{proposition}

Let us consider the following definition.
For a fixed $u\in\st^{\alpha}$ and $n\geq 1$, we define:
$$
\tilde{\xi}_n^{\alpha,u}=\left\{U^{\alpha,u}_{T_{n-1}}(A_n^{\alpha},O_n^{\alpha}) >\max_{(a,o)\in\mathcal{A}\times\mathcal{O}}U^{\alpha,u}_{T_{n-1}^{\alpha}}(a,o)-\frac{1}{2}\gamma\right\}. 
$$

\begin{remark} 
The event $\tilde{\xi}_n^{\alpha,u}$ is similar to the event $\xi_n^u$, defined in Section \ref{sec:proofs} for the model without communication bias and for the event $\xi_n^{\alpha,u}$ defined above. In the event $\tilde{\xi}_n^{\alpha,u}$ the system with initial social pressure $u$ chooses the $n$-th pair of actor/opinion between the most likely choices. This means that $(A_n^{\alpha},O_n^{\alpha})$ is one of the pairs in which the social pressure is close to the maximum the social pressure at time $T_{n-1}^{\alpha}$. By close we mean that this distance is at most $\frac{1}{2}\gamma$. 

The difference between these definitions of the model with and without bias is to ensure that we still can obtain a uniform lower bound for $\P(\tilde{\xi}_n^{\alpha,u})$ and that the properties described by Propositions \ref{prop2}, \ref{prop3} and \ref{prop4} still hold. The fact that the distance to the maximum social pressure is at most $\frac{1}{2}\gamma$ implies that for many $u\in \st^{\alpha}$, the event $\tilde{\xi}_n^{\alpha,u}$ is equal to $\xi_n^{\alpha,u}$, which is in particular helpful to prove that a sequence of events $\tilde{\xi}_n^{\alpha,u}, n=1,2,\ldots$, leads the process to $\mathcal{L}_{\alpha}$ (See Proposition \ref{prop3_2} below).
\end{remark}

With this definition, we have the following propositions.

\begin{proposition} \label{prop2_2}
For any $u\in \st^{\alpha}$, the following holds
$$
 \bigcap_{j=1}^N\tilde{\xi}_j^{\alpha,u}\subset\{-MN<U^{\alpha,u}_{T_N}<N\}.
$$
\end{proposition}

\begin{proposition}\label{prop3_2}
For any initial matrix $u\in\st^{\alpha}$, there exists $C(\alpha,M,N)\in \mathbb{N}$ such that
$$
  \left\{\bigcap_{j=1}^{C(\alpha,M,N)}\tilde{\xi}_j^{\alpha,u}\right\} \subset \left\{U^{\alpha,u}_{T_{C(\alpha,M,N)}^{\alpha}}\in\mathcal{L}^{\alpha} \right\}.
$$
\end{proposition}

\begin{proposition}\label{prop4_2}
    For all $ m\geq 1,$
    $$ \mathbb{P}\left(\bigcap^m_{j=1}\tilde{\xi}_j^{\alpha,u}\right)\geq (\zeta_{\alpha,\beta})^m,
    $$
    where
    $$
        \zeta_{\alpha,\beta}=\frac{e^{\frac{\beta}{2}\gamma}}{e^{\frac{\beta}{2}\gamma}+MN}\rightarrow1\text{, as }\beta\rightarrow+\infty.
    $$
\end{proposition}

The proof of Propositions \ref{prop2_2}, \ref{prop3_2} and \ref{prop4_2} follows as the proofs of Propositions \ref{prop2}, \ref{prop3} and \ref{prop4}, respectively. With these propositions, we can prove the following theorem.

\begin{theorem}\label{Theorem 1_2}
For any $\beta>0$, $0<\alpha<\frac{1}{M-1}$, and starting matrix $u\in\mathcal{S}^{\alpha}$:
    \begin{enumerate}
        \item The sequence $(T_m^{\alpha}:m\geq 1)$ of jumping times of the process $(U^{\alpha,\beta, u}_t)_t$ satisfies
    \[\mathbb{P}(\sup\{T_m^{\alpha,\beta,u}:m\geq 1\} = \infty)=1. \]
        \item The process $(U^{\alpha,\beta,u}_t)_t$ has a unique invariant probability measure $\mu^{\beta, \alpha}$.
    \end{enumerate}
\end{theorem}

The proof of Theorem \ref{Theorem 1_2} follows exactly as the proof of Theorem \ref{Theorem 1}.

Let us now define
\begin{multline*}
    \hat{\mathcal{L}}_{\alpha}=\{u\in\mathcal{S}^{\alpha}:\exists o\in\mathcal{O}, u(a,o) \in \mathbb{Z}, \forall a \in \A, \\
    0=u(a_1,o) < u(a_2,o) < \ldots < u(a_N,o), \{a_1,\ldots,a_N\}=\A,
    \\
    \text{ and } \forall p\neq o, u(a,p)=-\gamma u(a,o)\}.
\end{multline*}
Similarly as discussed in Remark \ref{rem:hatladder} for the model without communication bias, we have that $\mathcal{L}_{\alpha}\subset \hat{\mathcal{L}}_{\alpha}$ and for any $\hat{l} \in \hat{\mathcal{L}}_{\alpha}$, 
$$\{U^{\alpha,\hat{l}}_0(A_1^{\alpha},O_1^{\alpha})>0\} \subset \{ U^{\alpha,\hat{l}}_{T_1^{\alpha}} \in \hat{\mathcal{L}}_{\alpha}\}.
$$
Moreover, for any $\hat{l}\in\hat{\mathcal{L}}_{\alpha}$ and for a fixed $l\in \mathcal{L}_{\alpha}$,
$$
\mathbb{P}(U^{\alpha,\hat{l}}_0(A_1^{\alpha},O_1^{\alpha})>0)\geq \mathbb{P}(U^{\alpha,l}_0(A_1^{\alpha},O_1^{\alpha})>0)\geq \eta,
$$
where,
    $$
        \eta = \ddfrac{\sum^{N-1}_{j=1}e^{\beta j}}{\sum^{N-1}_{j=0}e^{\beta j}+MN}.
    $$

For any $u\in \st^{\alpha}$, $0<\alpha<\frac{1}{M-1}$ and $\beta \geq 0$, we define the skeleton process $(\tilde{U}_n^{\alpha,\beta,u})_{n=0,1,\ldots}$ given by $\tilde{U}_n^{\alpha,\beta,u}:= U_{T_n^{\alpha}}^{\alpha,\beta,u}$ for any $n=0,1,\ldots$. The unique invariant measure of this skeleton process is denoted $\tilde{\mu}^{\alpha,\beta}$. 
We have the following proposition. 
    
\begin{proposition}\label{theo2prop1_2}
     For any $0<\alpha<\frac{1}{M-1}, \beta>0, u\notin\hat{\mathcal{L}}_{\alpha}$, the following bound holds
    $$
        \Tilde{\mu}^{\alpha,\beta}(u)\leq \tilde{C}e^{-\beta(N-1)},
    $$
    where $\tilde{C}>0$ is a positive constant depending on $M,N$ and $\alpha$.
\end{proposition}

\begin{remark} \label{rem:jumpratelowerboundalpha}
Let $0<\alpha<\frac{1}{M-1}$. For any $a\in\A$ with $n_a\geq 1$, writing $n_a=(k-1)M+r$ with $k=\lceil n_a/M \rceil$ and $r\in\{1,\ldots,M\}$, and noting that $\max_{p\in\opn}c_p\geq \lceil n_a/M \rceil$, we have
$$
\max_{p\in\opn}u(a,p)\geq \left\lceil \frac{n_a}{M} \right\rceil(1+\gamma)-\gamma n_a
= k(M-1)\alpha+\gamma(M-r)\geq (M-1)\alpha.
$$
This bound is optimal: equality holds exactly for $k=1$ and $r=M$, that is, for the row of an actor that has heard exactly one expression of each opinion, all of whose entries equal $(M-1)\alpha$. Since any state of $\st^{\alpha}$ has at least one actor with $n_a\geq 1$, any state different from the zero matrix has at least one entry greater than or equal to $(M-1)\alpha$. Therefore the jump rate in any state other than the zero matrix is greater than or equal to $e^{\beta (M-1)\alpha}$. 
\end{remark}

The proof of Proposition \ref{theo2prop1_2} follows exactly as the proofs of Proposition \ref{theo2prop1}, taking into account Remark \ref{rem:jumpratelowerboundalpha} and recalling that $0\not\in \mathcal{S}^{\alpha}$.

For the next Theorem, we define the hitting time for any $0<\alpha<\frac{1}{M-1}$, $\beta>0, u\in\mathcal{S}^{\alpha}, \theta \subset\mathcal{S}^{\alpha}$:

$$
    R^{\alpha,\beta,u}(\theta):=\inf\{t\geq0:U^{\alpha,\beta,u}_t\in\theta\}.
$$
\begin{theorem}\label{Theorem 2_2} 
The following statements hold for any $0<\alpha<\frac{1}{M-1}$:
\begin{enumerate}
    \item There exists a constant $C>0$, depending on $M,N$ and $\alpha$, such that for any $\beta\geq 0$ the invariant probability measure $\mu^{\alpha,\beta}$ satisfies
            \[\mu^{\alpha,\beta}(\mathcal{L}_{\alpha})\geq 1-Ce^{-\beta(M-1)\alpha}. \]
\item For any fixed $\delta>0$
\[\sup_{u\in\mathcal{S}^{\alpha}\setminus\{0\}}\mathbb{P}\left(R^{\alpha,\beta, u}(\mathcal{L}_{\alpha})>e^{-\beta(M-1)\alpha(1-\delta)} \right)\rightarrow 0\text{, as }\beta\rightarrow +\infty. \]
    \end{enumerate}
\end{theorem}
 
The proof of Theorem \ref{Theorem 2_2} follows as the proof of Theorem \ref{Theorem 2}, with the following modifications: the lower bound for the jump rate of any non-null state is $e^{\beta(M-1)\alpha}$ by Remark \ref{rem:jumpratelowerboundalpha}, and the contribution of the zero matrix vanishes by Remark \ref{rem:Salphaproperties}, so that the deficit of invariant measure in Part 1 is bounded by $Ce^{\beta(N-1)}e^{-\beta(N-1+(M-1)\alpha)}=Ce^{-\beta(M-1)\alpha}$.

Finally, we have the following lemmas.

\begin{lemma} \label{theo3lemm1_2} For any $u \in \st^{\alpha}$,
    $$
\mathbb{P}(R^{\alpha,\beta,u}(\mathcal{L}_\alpha)>2\beta)
\leq  Ce^{-\frac{\beta}{2}\gamma},
$$
where $C>0$ is a positive constant depending on $M$, $N$ and $\alpha$.
\end{lemma}

\begin{lemma} \label{theo3lemm2_2}For any $t>0$, for any $o \in \opn$, the following holds for any $0<\alpha<\frac{1}{M-1}$.
\begin{enumerate}
    \item For any $l\in \mathcal{L}^{o}_{\alpha}$,
    $$
    \P(R^{\alpha,\beta,l}(\mathcal{C}^{-o}_{\alpha})>t) \geq  e^{-2tN^3(M+1)^3e^{-\frac{\beta}{2}\gamma}}.
    $$
    \item For any $u\in \mathcal{C}^{o}_{\alpha}$,
    $$
    \P(R^{\alpha,\beta,u}(\mathcal{C}^{-o}_{\alpha})\leq t) \leq (N^2M+2tN^3(M+1)^3)e^{-\frac{\beta}{2}\gamma}.
    $$
\end{enumerate}
\end{lemma}

The proof of Lemmas \ref{theo3lemm1_2} and \ref{theo3lemm2_2} follows as the proof of Lemmas \ref{theo3lemm1} and \ref{theo3lemm2}.

For any $o\in\mathcal{O}, l\in\mathcal{L}^{o}_{\alpha}$, 
let $c_{\alpha,\beta}$ be the positive real number such that
$$
    \mathbb{P}(R^{\alpha,\beta,l}(\mathcal{C}^{-o}_{\alpha})>c_{\alpha,\beta})=e^{-1}.
$$
We have the following corollary, whose proof follows exactly as the proof of Corollary \ref{theo3coro2}.

\begin{corollary}\label{theo3coro2_2}
    For any $0<\alpha<\frac{1}{M-1}$ and $\beta \geq 0$,
    $$
       c_{\alpha,\beta} \geq \frac{1}{2}N^{-3}(M+1)^{-3}e^{\frac{\beta}{2}\gamma}.
    $$
\end{corollary}

Finally, the next theorem gives us a metastable behavior when we look at the time it takes to observe a consensus transition when $\beta\rightarrow+\infty$.

\begin{theorem}\label{Theorem 3_2}
For any $0<\alpha<\frac{1}{M-1}$, there exists $\beta_0, C_1>0$ and  $C_2>0$, depending only on $\alpha$, $M$ and $N$, such that for any $\beta\geq \beta_0$, for any $o\in\mathcal{O}$ and for any  $u\in \mathcal{C}^o_\alpha$,
$$
\sup_{t\geq 0}\left|\P\left(\frac{R^{\alpha,\beta,u}(\mathcal{C}^{-o}_{\alpha})}{\mathbb{E}(R^{\alpha,\beta,u}(\mathcal{C}^{-o}_{\alpha}))}>t\right)-e^{-t}\right| \leq C_1\beta^3e^{-C_2 \beta}.
$$
Moreover, for any $u,v \in \mathcal{C}^o_{\alpha}$,
$$
\left|\frac{\mathbb{E}(R^{\alpha,\beta,u}(\mathcal{C}^{-o}_{\alpha}))}{\mathbb{E}(R^{\alpha,\beta,v}(\mathcal{C}^{-o}_{\alpha}))}-1\right| \leq C_1\beta^3e^{-C_2 \beta}.
$$
\end{theorem}

The proof of Theorem \ref{Theorem 3_2} follows exactly as the proof of Theorem \ref{Theorem 3}.

\section*{Acknowledgements}

This work was produced as part of the activities of FAPESP Research, Innovation and Dissemination Center for Neuromathematics (grant \# 2013/ 07699-0 , S.Paulo Research Foundation (FAPESP)). 
KL was successively supported by FAPESP fellowship (grant 2022/07386-0 and 2023/12335-9).

This article is dedicated to the memory of Antonio Galves, who supervised FP research internship in 2023 where this work was initiated and introduced the research subjects of the present article to both authors with enthusiasm. 
The authors thank the NeuroMat group at USP for support.

\printbibliography

\end{document}